\documentclass[12pt,a4paper,dvipsnames]{article}

\usepackage[utf8]{inputenc}
\usepackage[english]{babel}
\usepackage{amsmath}
\usepackage{amsthm}
\usepackage{amssymb}
\usepackage{geometry}
\usepackage{a4wide}
\usepackage{bm}
\usepackage{microtype}

\usepackage{pgf,tikz,pgfplots}
\pgfplotsset{compat=1.15}
\usetikzlibrary{shapes.geometric}
\usetikzlibrary{patterns}
\usepackage{subcaption}

\def\normaledge{1.2}
\definecolor{edgeblack}{rgb}{0.25,0.25,0.25}
\definecolor{vertexblack}{rgb}{0.2,0.2,0.2}

\usepackage{enumitem}
\setlist[enumerate]{itemsep=0mm}
\usepackage{xcolor}
\newcommand\myshade{85}
\colorlet{myurlcolor}{Aquamarine}
\usepackage{hyperref}
\hypersetup{
  linkcolor  = black,
  citecolor  = black,
  urlcolor   = myurlcolor!\myshade!black,
  colorlinks = true,
}
\usepackage{tikz, graphicx, standalone, caption, subcaption, wrapfig}

\theoremstyle{definition}
\newtheorem{theorem}{Theorem}[section]

\newtheorem{lemma}[theorem]{Lemma}
\newtheorem*{lemma*}{Lemma}
\newtheorem*{lemma''*}{``Lemma''}

\newtheorem{corollary}[theorem]{Corollary}
\newtheorem{question}[theorem]{Question}

\newtheorem{example}[theorem]{Example}
\newtheorem*{problem*}{Problem}
\newtheorem{definition}[theorem]{Definition}
\theoremstyle{remark}
\newtheorem*{claim*}{Claim}

\DeclareMathOperator{\rk}{rk}
\DeclareMathOperator{\rank}{rank}

\DeclareMathOperator{\Span}{span}

\DeclareMathOperator{\dof}{dof}

\newcommand{\rigiditymatroid}[1]{\mathcal{R}_d(#1)}

\newcommand{\RR}{\mathbb{R}}

\newcommand{\CC}{\mathbb{C}}
\newcommand{\ZZ}{\mathbb{Z}}
\newcommand{\QQ}{\mathbb{Q}}

\newcommand{\varpsi}{\psi}

\newcommand{\norm}[1]{\left\lVert#1\right\rVert}

\newcommand{\affineclass}{\mathcal{A}}
\newcommand{\gaussfiber}{F}

\newcommand {\p}{p}
\newcommand {\q}{q}
\newcommand {\x}{x}

\newcommand{\Rd}{\mathcal{R}_d}

\title{\bf Globally rigid graphs are fully reconstructible}
\author{D\'aniel Garamv\"olgyi\thanks{Department of Operations Research,
E\"otv\"os University, and the 
MTA-ELTE Egerv\'ary Research Group on Combinatorial Optimization,
P\'azm\'any P\'eter s\'et\'any 1/C, 1117 Budapest, Hungary.
e-mail: {\tt daniel.garamvolgyi@ttk.elte.hu}}
\and
Steven J. Gortler\thanks{Harvard School of Engineering and Applied Sciences,  33 Oxford St., Cambridge,
MA 02138 USA. e-mail: {\tt sjg@cs.harvard.edu}}
\and
Tibor Jord{\'a}n\thanks{Department of Operations Research,
E\"otv\"os University, and the 
MTA-ELTE Egerv\'ary Research Group on Combinatorial Optimization,
P\'azm\'any P\'eter s\'et\'any 1/C, 1117 Budapest, Hungary.
e-mail: {\tt tibor.jordan@ttk.elte.hu}} 
}

\date{June 2, 2022}

\begin{document}
\maketitle

\begin{abstract}
A $d$-dimensional framework is a pair $(G,p)$, where $G=(V,E)$ is a graph and $p$ is
a map from $V$ to  $\RR^d$. The length of an edge $uv\in E$ in $(G,p)$ is the distance
between $p(u)$ and $p(v)$. The framework is said to be globally rigid in
$\RR^d$ if the graph $G$ and its edge lengths uniquely determine $(G,p)$, up to congruence.
%
%
%every other $d$-dimensional framework $(G,q)$, in which
%the corresponding edge lengths are the same, is congruent to $(G,p)$.
A graph $G$ is called globally rigid in $\RR^d$ if every $d$-dimensional
generic framework $(G,p)$ is globally rigid.

In this paper, we consider the 
problem of reconstructing a graph from the set of
edge lengths arising from a generic framework.
Roughly speaking, 
a graph $G$ is strongly reconstructible in
$\CC^d$ 
if the set of (unlabeled) edge lengths of any generic framework $(G,p)$ in $d$-space, along with the number of vertices of $G$, uniquely determine both $G$ and the association between the edges of $G$ and the set of edge lengths. 
%if it is uniquely determined by the set of (unlabeled) edge lengths of any generic framework $(G,p)$ in $d$-space, along with the number of its vertices. 
It is known 
that if $G$ is  globally rigid in $\RR^d$ on at least $d+2$ vertices, then it is strongly reconstructible in $\CC^d$.
We 
strengthen this result and show that under the same conditions, $G$ is in fact fully reconstructible in $\CC^d$,
which means that the set of edge lengths alone is sufficient to uniquely reconstruct $G$, without any constraint on the
number of vertices (although still under the assumption that the edge lengths come from a generic realization).
%
%The characterizations of globally rigid graphs and strongly reconstructible graphs
%are known for $d=1,2$ and are major open questions in higher dimensions.

As a key step in our proof,
we also
prove that if $G$ is globally rigid in $\RR^d$ on at least $d+2$ vertices, then the $d$-dimensional generic
rigidity matroid of $G$ is connected. 
%This result is essentially a generalization Hendrickson's necessary condition for global rigidity and gives rise to a new combinatorial necessary condition. 
%It leads to new families of graphs which are
%$(d+1)$-connected, redundantly rigid, but not globally rigid.
Finally, we provide new families of fully reconstructible graphs and use them to answer some questions regarding unlabeled reconstructibility posed in recent papers.
%\cite{GJ,GTT}.
\end{abstract}

\section{Introduction}

A \textit{($d$-dimensional) framework} is a pair $(G,p)$ where $G = (V,E)$ is a graph and $p : V \rightarrow \RR^d$ is a map
	that assigns a point in $\RR^{d}$ to each vertex of $G$. The {\it length} of an edge $uv$
	in $(G,p)$ is the Euclidean distance between $p(u)$ and $p(v)$. We also call $(G,p)$
	a {\it realization} of $G$ in $\RR^{d}$. The framework is {\it generic}
	if the set of the coordinates of its points is algebraically independent over $\QQ$.
	We say that a $d$-dimensional framework $(G,p)$ is {\it globally rigid}
	%in $\RR^d$ 
	if every other realization $(G,q)$ of $G$ in $\RR^d$ in which corresponding edges have the same length
is congruent to $(G,p)$. That is, the graph $G$ and its edge lengths in $(G,p)$ uniquely determine the pairwise distances of all vertices in $(G,p)$. It is known that for generic $d$-dimensional frameworks, global rigidity depends only on the graph $G$, for all $d\geq 1$. We say that $G$ is \emph{(generically) globally rigid in $\RR^d$} if every (equivalently, if some) generic realization of $G$ in $\RR^d$ is globally rigid.
%We shall consider and solve various questions concerning the global rigidity of generic frameworks.
In the rest of this section we give a brief overview of our main results. Most of the definitions and more details are given
in the next section.

%Unlabeled rigidity 
In the context of rigidity theory, unlabeled reconstruction is the study of what 
combinatorial and geometric information is determined by 
the (multi)-set of edge lengths arising from some $d$-dimensional framework $(G,p)$.
%$G$ and point configuration $p$. 
In~\cite{GTT}, it was shown
that if $(G,p)$ is a generic globally rigid framework with $n$ vertices in $\RR^d$, where $n \geq d+2$,
then there can be no distinct (not necessarily generic) realization $(H,q)$ of 
any graph $H$ with $n$ vertices
in $\RR^d$ that
produces the same edge lengths, up to trivialities. This result
is essentially tight: if $H$ is allowed to have more vertices or if
$G$ is not globally rigid, then
$p$ cannot be determined.
To prove this result, it was sufficient
to study the following, related graph reconstruction question.

%The $d$-dimensional measurement variety $M_{d,G}$ of a graph $G$ is the Zariski-closure of the set of all vectors arising as the squared edge lengths of $d$-dimensional realizations of $G$.
% It turns out that various geometric properties of the measurement variety are closely connected to the local and global rigidity properties of $G$. Thus, 
%It is natural to look for conditions under which the measurement variety itself determines the underlying graph. 
%As is often the case when dealing
%with varieties, it is more convenient %(and still valid) 
%to do this
%analysis in the complex setting.
We say that a graph $G$ is {\it strongly reconstructible in
$\CC^d$} if whenever a pair of generic frameworks $(G,p)$ and $(H,q)$ in $\CC^d$
have the same set of edge lengths and the same number of vertices, we have that
$G$ is isomorphic to $H$ and the corresponding edges have the same length. 
%This means, roughly, that whenever there is a bijection $\varpsi: E(G) \rightarrow E(H)$ that sends $M_{d,G}$ to $M_{d,H}$, where $H$ has the same number of vertices as $G$, then $G$ and $H$ are isomorphic and $\varpsi$ is consistent with an isomorphism between them (for detailed definitions see Subsection \ref{subsec:unlabeled} and Theorem \ref{theorem:strongmeasurereconstructibility}).
In~\cite{GTT}, it was shown that for $d \geq 2$, globally rigid graphs in $\RR^d$ on at least $d + 2$ vertices
are strongly reconstructible in $\CC^d$.
Although this %strong reconstructibility 
result was sufficient to answer
the original question of~\cite{GTT}, as a pure reconstruction question, the 
dependence on $n$ remains unsatisfying.
To this end, 
we call a graph $G$ {\it fully reconstructible in $\CC^d$} if whenever a pair of generic frameworks $(G,p)$ and $(H,q)$ in $\CC^d$
have the same set of edge lengths, but not necessarily the same number of vertices, we have that
$G$ is isomorphic to $H$ and the corresponding edges have the same length. We stress that in the above definition we require both $(G,p)$ and $(H,q)$ to be generic. This, or some other non-degeneracy condition, is necessary since the edge lengths of $(G,p)$ can be realized by an arbitrary forest $H$ on $|E(G)|$ edges (but the realization $(H,q)$ obtained in this way will usually be non-generic). 
%In the case of strong reconstructibility, we can omit this genericity condition when $G$ is rigid in $\RR^d$, see \cite[Theorem 3.6]{GJ}.
%As with strong reconstructibility, full reconstructibility of $G$ can also be interpreted as a uniqueness condition on the measurement variety of $G$, see Theorem \ref{theorem:strongmeasurereconstructibility}.

As we shall see, both strong reconstructibility and full reconstructibility can be (perhaps more naturally) stated as a uniqueness condition not on the generic realizations of $G$, but rather on a certain complex variety associated to $G$. This variety, which we call the $d$-dimensional measurement variety of $G$ and denote by $M_{d,G}$, is the Zariski-closure of the set of all vectors arising as the squared edge lengths of $d$-dimensional realizations of $G$.
Roughly speaking, $G$ is strongly reconstructible in $\CC^d$ if whenever there is a bijection $\varpsi: E(G) \rightarrow E(H)$ that sends $M_{d,G}$ to $M_{d,H}$, where $H$ has the same number of vertices as $G$, then $G$ and $H$ are isomorphic and $\varpsi$ is consistent with an isomorphism between them (for detailed definitions see Subsection \ref{subsec:unlabeled} and Theorem \ref{theorem:strongmeasurereconstructibility}). Full reconstructibility has a similar characterization, with the only difference being that we do not require $H$ to have the same number of vertices as $G$. Thus, strong reconstructibility in $\CC^d$ asserts that $G$, in a sense, is uniquely determined by $n,d$ and $M_{d,G}$, while full reconstructibility means that $G$ is uniquely determined by $d$ and $M_{d,G}$. 

In \cite{GTT}, the question was posed
whether global rigidity in $\RR^d$ implies full reconstructibility in $\CC^d$.
For $d=2$, the question was answered affirmatively in \cite{GJ}.
In this paper, as one of our main results, 
we substantially strengthen the previous results and show that for all $d\geq 2$,
$$\hbox{{\em if $G$ is globally rigid in $\RR^d$ on $n\geq d+2$ vertices, then $G$ is fully reconstructible in $\CC^d$.}}$$

Results in~\cite{GJ} show that for $d = 2$, this is tight:
if a graph is not globally rigid in $\RR^2$, then
it is not even strongly reconstructible in $\CC^2$. The $d=1$ case is slightly different but also fully
characterizable using $3$-connectivity.
For $d \geq 3$,  a characterization of full reconstructibility 
seems more elusive.
In particular, we show that global rigidity in $\RR^d$ is not 
necessary for full reconstructibility in $\CC^d$. 
%Towards such a characterization, 
We also prove some positive and negative results
 regarding possible sufficient and possible necessary 
conditions for strong and full reconstructibility. 
These answer 
a number
of questions that were posed in \cite{GTT}
and \cite{GJ}.

To prove our result on full reconstructibility we also prove a combinatorial theorem, interesting on its own right. We say that a graph is \emph{$\Rd$-connected} if its $d$-dimensional (generic) rigidity matroid is connected (see the next section for detailed definitions). 
A combinatorial characterization of globally rigid graphs in $\RR^d$ is known only for $d=1,2$, and is a major open problem for $d\geq 3$.
In higher dimensions, Hendrickson's theorem (see Theorem \ref{theorem:hendrickson}) gives combinatorial necessary conditions that link global rigidity to connectivity
and local rigidity properties of $G$. 
In this paper, as another main result, we strengthen one of Hendrickson's necessary conditions (redundant rigidity) by proving that for all $d \geq 1$,
$$\hbox{{\em if $G$ is globally rigid in $\RR^d$ on $n\geq d+2$ vertices, then it is $\Rd$-connected.}}$$ 
This result may lead to a better understanding of higher dimensional global rigidity. In particular, we use it to find new examples of so-called \emph{$H$-graphs}, graphs that satisfy Hendrickson's conditions but are not globally rigid.

The rest of the paper is laid out as follows. In Section \ref{section:prelims}, we recall the definitions and results from rigidity theory and algebraic geometry that we shall use throughout the paper. Section \ref{section:main} contains our main results: after making some structural observations about the measurement variety of graphs, we show that globally rigid graphs in $\RR^d$ on at least $d+2$ vertices are $\Rd$-connected (Theorem \ref{theorem:mconnected}) and, for $d \geq 2$, fully reconstructible in $\CC^d$ (Theorem \ref{theorem:fullyreconstructible}). In Section \ref{section:examples}, we illustrate the results of the previous section with several examples. In particular, we use Theorem \ref{theorem:mconnected} to give new examples of $H$-graphs. We also answer questions from \cite{GJ} and \cite{GTT} related to the unlabeled reconstructibility problem, as well as pose new open questions. Finally, in Section \ref{section:mseparable} we prove some new results regarding $\Rd$-connected graphs.

\section{Preliminaries}\label{section:prelims}

We start by fixing some conventions. In the following, graphs will be understood to be simple, that is, without parallel edges and loops. To avoid some trivialities, we shall also implicitly assume that every graph considered has at least one edge (and thus at least two vertices). For a graph $G = (V,E)$, we shall use $\RR^E$ and $\CC^E$ to denote the $|E|$-dimensional real (complex, respectively) Euclidean space with axes labelled by the edges of $G$. We shall also often refer to the \emph{configuration spaces} $\RR^{nd}$ and $\CC^{nd}$, where $n$ denotes the number of vertices of $G$ and $d \geq 1$ is some dimension. We shall really think of these spaces as $(\RR^d)^V$ and $(\CC^d)^V$, i.e.\ as $n$-tuples of $d$-dimensional vectors, indexed by the vertices of $G$. Nonetheless, as it is less cumbersome, we shall use the notation $\RR^{nd}$ and $\CC^{nd}$. 

\subsection{Real and complex frameworks}

Let $G = (V,E)$ be a graph
on $n$ vertices and $d \geq 1$ some fixed integer. A \textit{$d$-dimensional
  realization} of $G$ is a pair $(G,p)$ where 
%$\p = (p_1,...,p_n)$ is a point in $\RR^{nd}$, or equivalently, 
$p : V \rightarrow \RR^d$ maps the vertices of $G$ into Euclidean space. We call such a map $p$ (or equivalently, a point $p = (p_v)_{v \in V}$ in $\RR^{nd}$) a \textit{configuration} and we say that the pair $(G,p)$ is a \textit{framework}. Two $d$-dimensional frameworks $(G,p)$ and $(G,q)$ are \textit{equivalent} if $\norm{p(u) - p(v)} = \norm{q(u) - q(v)}$ for every edge $uv \in E$, and \textit{congruent} if the same holds for every pair of vertices $u,v \in V$. Here $\norm{\cdot}$ denotes the Euclidean norm.

A framework is \textit{(locally) rigid} if every continuous motion of the
vertices which preserves the edge lengths takes it to a congruent framework,
and \textit{globally rigid} if every equivalent framework is congruent to it. 

We say that a configuration $\p \in \RR^{nd}$ is \textit{generic} if its $n \cdot
d$ coordinates are algebraically independent over $\QQ$. It is known that 
in any fixed dimension $d$, both local and global rigidity are generic properties of the underlying graph, in the sense that either every generic $d$-dimensional framework is locally/globally rigid or none of them are (see  
%\cite{asimowroth} and 
\cite{connelly,GHT}). Thus, we say that a graph is \emph{rigid} (respectively \emph{globally rigid}) in $d$ dimensions if every (or equivalently, if some) generic $d$-dimensional realization of the graph is rigid (resp.\ globally rigid).

Let $G = (V,E)$ be a graph on $n$ vertices and $d \geq 1$. The function $m_{d,G} : \RR^{nd} \rightarrow \RR^E$ mapping each realization of $G$ to the sequence of its Euclidean squared edge lengths is called the \emph{rigidity map} or \emph{edge measurement map} of $G$ in $d$ dimensions. That is, for a $d$-dimensional realization $(G,p)$ of $G$, the coordinate of $m_{d,G}(\p)$ corresponding to the edge $uv \in E$ is $\norm{p(u) - p(v)}^2$. 
%Given a $d$-dimensional framework $(G,\p)$, we say that $m_{d,G}(\p)$ is the \textit{edge measurement sequence} of $\p$.

Analogously to the real case, we define a $d$-dimensional \textit{complex framework} to be a pair $(G,p)$, where $G = (V,E)$ is a graph and $p : V \rightarrow \CC^d$ is a complex mapping.
Given a framework $(G,p)$ in $\CC^d$ and a pair of vertices $u,v$ in $G$, we define the \emph{complex squared distance} of $u$ and $v$ in $(G,p)$ by
\begin{equation*}
    m_{uv}(\p) = \left(p(u) - p(v)\right)^{T} \cdot \left(p(u) - p(v)\right) =  \sum_{k=1}^d (p(u)_k - p(v)_k)^2,
\end{equation*}
where $k$ indexes over the $d$ dimension-coordinates. Note that in this definition we do not use conjugation, and thus this is \emph{not} the usual distance of $p(u)$ and $p(v)$. We also note that the mapping $m_{uv} : \CC^{nd} \rightarrow \CC$ depends on $G$ and $d$, but since these will always be clear from the context, we shall omit them from our notation.  

For an edge $e = uv$ of $G$, we say that $m_{uv}(p)$ is the \emph{complex squared length} of the edge in $(G,p)$. For real frameworks, this coincides with the
usual (Euclidean) squared length, so we can extend $m_{d,G}$ to a $\CC^{nd} \rightarrow \CC^E$ function by letting
\begin{equation*}
m_{d,G}(p) = \big(m_{uv}(p)\big)_{uv \in E}.
\end{equation*}
We say, as in the real case, that two frameworks $(G,\p)$ and $(G,\q)$ are
\textit{equivalent} if $m_{d,G}(\p) = m_{d,G}(\q)$, and they
are \textit{congruent} if $m_{d,K_V}(p) = m_{d,K_V}(q)$, where $K_V$ is the complete graph on the vertex set $V$. A  
configuration  $\p \in \CC^{nd}$ is, again, generic, if the coordinates of $\p$
are algebraically independent over $\QQ$. A point $\p \in \RR^{nd}$ is generic as a real configuration precisely if it is generic as a complex one. 

Although we will not explicitly need them, we note that one can define the analogues of rigidity and global rigidity for complex frameworks. It turns out that, as in the real case, the (global) rigidity of generic complex frameworks only depends on the underlying graph, and the corresponding graph property of being ``(globally) rigid in $\CC^d$'' is equivalent to being (globally) rigid in $\RR^d$, see \cite{GT,GTT}.

\iffalse
Using these notions one can define the analogues of rigidity and global
rigidity for complex frameworks. It turns out that, as in the real case, the (global) rigidity of generic complex frameworks only depends on the underlying graph, and the graph properties obtained in this way  coincide with their real counterpart.

\begin{theorem}\label{theorem:complexrigidity} \cite{GT,GTT}
Complex rigidity and global rigidity are generic properties and a graph $G$ is rigid (respectively globally rigid) in $\CC^d$ if and only if it is rigid (resp. globally rigid) in $\RR^d$. 
\end{theorem}

In light of Theorem \ref{theorem:complexrigidity}, the terms ``(globally) rigid in $\RR^d$'' and ``(globally) rigid in $\CC^d$'' are interchargeable for a graph. We shall always use the former to emphasize that, although we often work in the complex setting, our results are related to the standard (real) notion of global rigidity.
\fi

It follows from the definitions  that globally rigid graphs are rigid. 
The following much stronger necessary conditions of global rigidity 
 are due to Hendrickson \cite{hend}.
We say that a graph is \textit{redundantly rigid} in a given dimension if it
remains rigid after deleting any edge. A graph is \emph{$k$-connected} for some $k
\geq 2$ if it has at least $k+1$ vertices and it remains connected after
deleting any set of less than $k$ vertices. 

\begin{theorem}\cite{hend}\label{theorem:hendrickson}
Let $G$ be a graph on at least $d+2$ vertices for some $d \geq 1$. Suppose that $G$ is globally rigid in $\RR^d$. Then $G$ is $(d+1)$-connected and redundantly rigid in $\RR^d$.
\end{theorem}

In $d = 1,2$ dimensions the conditions of Theorem \ref{theorem:hendrickson} are, in fact, sufficient for global rigidity \cite{JJconnrig}. 
This fails in the $d \geq 3$ case and a combinatorial characterization of globally rigid graphs in these dimensions is a major open question.

\subsection{The rigidity matrix and the rigidity matroid}

The rigidity matroid of a graph $G$ is a matroid defined on the edge set
of $G$ which reflects the rigidity properties of all generic realizations of
$G$. For a general introduction to matroid theory we refer the reader to \cite{oxley}.
Let $(G,p)$ be a realization of a graph $G=(V,E)$ in $\RR^d$.
The {\it rigidity matrix} of the framework $(G,p)$
is the matrix $R(G,p)$ of size
$|E|\times d|V|$, where, for each edge $v_iv_j\in E$, in the row
corresponding to $v_iv_j$,
the entries in the $d$ columns corresponding to vertices $v_i$ and $v_j$ contain
the $d$ coordinates of
$(p(v_i)-p(v_j))$ and $(p(v_j)-p(v_i))$, respectively,
and the remaining entries
are zeros. In other words, it is $1/2$ times the Jacobian of the rigidity map $m_{d,G}$. 
%See \cite{Whlong} for more details.
The rigidity matrix of $(G,p)$ defines
the {\it rigidity matroid}  of $(G,p)$ on the ground set $E$
by linear independence of rows. It is known that any pair of generic frameworks
$(G,p)$ and $(G,q)$ have the same rigidity matroid.
We call this the $d$-dimensional {\it rigidity matroid}
${\cal R}_d(G)=(E,r_d)$ of the graph $G$.
We can define the rigidity matrix $R(G,p)$ for complex frameworks in the same way as in the real case. This, again, allows us to define the rigidity matroid of the framework. It is not difficult to show that the rigidity matroid of a generic framework in $\CC^d$ is, again, the $d$-dimensional rigidity matroid $\mathcal{R}_d(G)$.
%, see \cite{???}. 

We denote the rank of ${\cal R}_d(G)$ by $r_d(G)$.
A graph $G=(V,E)$ is {\it $\Rd$-independent} if $r_d(G)=|E|$ and it is an {\it $\Rd$-circuit}  if it is not $\Rd$-independent, but every proper 
subgraph $G'$ of $G$ is $\Rd$-independent. We note that in the literature such graphs are sometimes called $M$-independent in $\RR^d$ and $M$-circuits in $\RR^d$, respectively. An edge $e$ of $G$ is an {\it $\Rd$-bridge in $G$}
if  $r_d(G-e)=r_d(G)-1$ holds. Equivalently, $e$ is an $\Rd$-bridge in $G$ if it is not contained in any subgraph of $G$ that is an $\Rd$-circuit.

Gluck characterized rigid graphs in terms
of their rank.

\begin{theorem}\label{thm:gluck}
\cite{Gluck}
\label{combrigid}
Let $G=(V,E)$ be a graph with $|V|\geq d+1$. Then $G$ is rigid in $\RR^d$
if and only if $r_d(G)=d|V|-\binom{d+1}{2}$.
\end{theorem}

Let ${\cal M}$ be a matroid on ground set $E$ with rank function $r$. 
We can define a relation on the pairs of elements of $E$ by
saying that $e,f\in E$ are
equivalent if $e=f$ or there is a circuit $C$ of ${\cal M}$
with $\{e,f\}\subseteq C$.
This defines an equivalence relation. The equivalence classes are 
the {\it connected components} of ${\cal M}$.
The matroid is said to be {\it connected} if there is only one equivalence class, and {\it separable} otherwise.
We shall use the fact that ${\cal M}$ is separable
if and only if there is a partition $E=E_1\cup E_2$ of $E$ into two non-empty subsets
for which 
\begin{equation*}
r({\cal M})=r({\cal M}_1)+r({\cal M}_2),
\end{equation*}
holds, where ${\cal M}_i$ denotes
the restriction of ${\cal M}$ to $E_i$, $i=1,2$.

Given a graph $G=(V,E)$, the subgraphs induced by the edge sets of the connected components
of ${\cal R}_d(G)$ are the
{\it $\Rd$-connected components} of $G$.
The graph is said to be {\it $\Rd$-connected} if ${\cal R}_d(G)$ is connected,
and {\it $\Rd$-separable} otherwise. See Figure \ref{figure:mseparable1} for an example of an $\mathcal{R}_3$-separable graph.
%Thus an edge $e$ is a $d$-bridge if and only if $E(H_i)=\{e\}$ for some $1\leq i\leq t$.

In one and two dimensions, $\Rd$-connectivity has close ties with (global) rigidity, as shown by the following result.

\begin{theorem}\label{theorem:JacksonJordan}\cite{JJconnrig}
Let $G$ be a graph without isolated vertices and $d \in \{1,2\}$. Then
\begin{enumerate}[label=\emph{(\alph*)}]
    \item If $G$ is globally rigid in $\RR^d$ on at least $d+2$ vertices, then it is $\Rd$-connected.
    \item If $G$ is $\Rd$-connected, then it is redundantly rigid in $\RR^d$.
\end{enumerate}
\end{theorem}

It is known that part \emph{(b)} of Theorem \ref{theorem:JacksonJordan} is not true in $d \geq 3$ dimensions. On the other hand, we shall show that part \emph{(a)} remains valid in $d$ dimensions for all $d \geq 3$ (Theorem \ref{theorem:mconnected}). We note that this is essentially a strengthening of the second part of Hendrickson's theorem (Theorem \ref{theorem:hendrickson}), which can be seen as follows. A graph is redundantly rigid in $\RR^d$ if and only if it is rigid in $\RR^d$ and contains no $\Rd$-bridges, and $\Rd$-connected graphs contain no $\Rd$-bridges (indeed, $e$ is an $\Rd$-bridge in $G$ if and only if $\{e\}$ is an $\Rd$-connected component of $G$). Since a globally rigid graph is clearly rigid, Theorem \ref{theorem:mconnected} implies that such a graph (on at least $d+2$ vertices) is redundantly rigid as well.

%Since an edge $e$ is an $\Rd$-bridge if and only if $\{e\}$ is the edge set of an $\Rd$-connected component of $G$, and a rigid graph is redundantly rigid if and only if it has no $\Rd$-bridges, part \emph{(a)} of Theorem \ref{theorem:JacksonJordan} is a strengthening of the second part of Theorem \ref{theorem:hendrickson} in the $d \leq 2$ case. We shall show (Theorem \ref{theorem:mconnected}) that this strengthening remains true in the $d \geq 3$ case. On the other hand, it is known that part \emph{(b)} of Theorem \ref{theorem:JacksonJordan} is not true in $d \geq 3$ dimensions.

Let $G = (V,E)$ be a graph on $n$ vertices and $(G,p)$ a framework in $\CC^d$. The elements of $\ker(R(G,p)) \subseteq \CC^{nd}$ are the \emph{infinitesimal motions} of $(G,p)$, while the elements of $\ker(R(G,p)^T) \subseteq \CC^E$ are the \emph{equilibrium stresses} (or \emph{stresses}, for short) of $(G,p)$. We shall also use the notation $S(G,p) = \ker(R(G,p)^T)$. Let us denote the column space of $R(G,p)$ by $\Span(R(G,p))$. By basic linear algebra, $S(G,p)$ is the orthogonal complement\footnote{Note that here, as well as later on, we shall say that two vectors $x,y \in \CC^E$ are orthogonal if $x^Ty = 0$; that is, we work with the symmetric bilinear form $\langle x,y \rangle = x^Ty$, and \emph{not} the hermitian bilinear form  $\langle x,y \rangle = x^*y$.}
of $\Span(R(G,p))$ in $\CC^E$. Since the elements of $S(G,p)$ capture the row dependences of $R(G,p)$, we have that $G$ is $\Rd$-independent if and only if for every generic realization $(G,p)$ in $\CC^d$, we have $S(G,p) = \{0\}$, and $G$ is an $\Rd$-circuit if and only if every generic realization $(G,p)$ in $\CC^d$ has a unique (up to scalar multiple) non-zero equilibrium stress $\omega$ and $\omega$ is non-zero on every edge of $G$.

Suppose that $G$ is $\Rd$-separable and let $E = E_1 \cup E_2$ be a partition of $E$ into non-empty subsets such that for the graphs $G_i$ induced by $E_i, i=1,2$, we have $r_d(G) = r_d(G_1) + r_d(G_2)$. It is not difficult to see that in this case for any generic realization $(G,p)$ in $\CC^d$, we have $\Span(R(G,p)) = \Span(R(G_1,p)) \oplus \Span(R(G_2,p))$ as linear subspaces of $\CC^E$, under the identification $\CC^E = \CC^{E_1} \times \CC^{E_2}$. This also implies $S(G,p) = S(G_1,p) \oplus S(G_2,p)$ under the same identification.

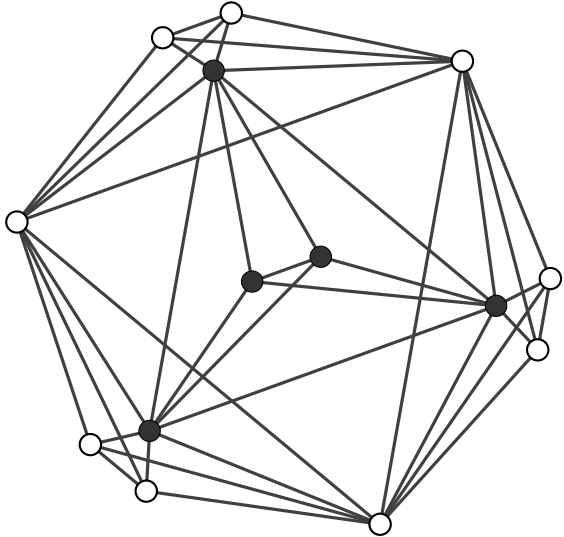
\begin{figure}[ht!]
    \centering
        \begin{tikzpicture}[x = 1cm, y = 1cm, scale = 4]
            \coordinate (B) at (0,0.7);
            \coordinate (B') at (0,0.9);
            \coordinate (C) at (0.12,.86);
            \coordinate (D) at (-0.12,.86);
            \foreach \k in {1,2,3}
            {
                \pgfmathtruncatemacro\deg{120*\k + 20}
                \pgfmathtruncatemacro\bigdeg{120*\k + 80}
                \coordinate (B\k) at ([rotate around={\deg:(0,0)}]B) {};
                \coordinate (C\k) at ([rotate around={\deg:(0,0)}]C) {};
                \coordinate (D\k) at ([rotate around={\deg:(0,0)}]D) {};
                \coordinate (E\k) at ([rotate around={\bigdeg:(0,0)}]B') {};
                \draw [line width=\normaledge, color=edgeblack] (B\k) -- (C\k);
                \draw [line width=\normaledge, color=edgeblack] (C\k) -- (D\k);
                \draw [line width=\normaledge, color=edgeblack] (D\k) -- (B\k);
                \draw [line width=\normaledge, color=edgeblack] (E\k) -- (B\k);
                \draw [line width=\normaledge, color=edgeblack] (E\k) -- (C\k);
                \draw [line width=\normaledge, color=edgeblack] (E\k) -- (D\k);
            }
            
            \coordinate (O1') at (-0.12,0);
            \coordinate (O2') at (0.12,0);
            \coordinate (O1) at ([rotate around={20:(0,0)}]O1') {};
            \coordinate (O2) at ([rotate around={20:(0,0)}]O2') {};
            \draw [line width=\normaledge, color=edgeblack] (O1) -- (O2);            
        
            \foreach \k in {1,2,3}
            {
                \pgfmathtruncatemacro\kplus{Mod(\k,3) + 1}
                \draw [line width=\normaledge, color=edgeblack] (E\k) -- (B\kplus);
                \draw [line width=\normaledge, color=edgeblack] (E\k) -- (C\kplus);
                \draw [line width=\normaledge, color=edgeblack] (E\k) -- (D\kplus);
                \draw [line width=\normaledge, color=edgeblack] (E\k) -- (E\kplus);
                \draw [line width=\normaledge, color=edgeblack] (B\k) -- (B\kplus);
                
                \draw [line width=\normaledge, color=edgeblack] (B\k) -- (O1);
                \draw [line width=\normaledge, color=edgeblack] (B\k) -- (O2);
            }

            \draw [fill=vertexblack] (O1) circle (1pt);
            \draw [fill=vertexblack] (O2) circle (1pt);
            
            \foreach \k in {1,2,3}
            {   
                \draw [fill=vertexblack] (B\k) circle (1pt);
                \draw [fill=white, thick] (C\k) circle (1pt);
                \draw [fill=white, thick] (D\k) circle (1pt);
                \draw [fill=white, thick] (E\k) circle (1pt);
            }
        \end{tikzpicture}
    \caption{A $3$-connected, redundantly rigid and $M$-separable graph in $\RR^3$.
    This graph satisfies $r_3(G)= 36 = 27 + 9 =r_3(G^o)+r_3(K_5)$, where
$G^o$ is the outer ring of $K_5$'s and $K_5$ is the subgraph induced by the black (filled) vertices.}
    \label{figure:mseparable1}
\end{figure}

We close this section by recalling the well-known coning operation on a graph. Given a graph $G$, the \emph{cone} of $G$ is obtained by adding a new vertex $v$ to $G$, along with new edges from $v$ to every vertex of $G$. Coning provides a transfer between various rigidity properties in $d$ and $d+1$ dimensions, as shown by the following results.

\begin{theorem}\label{thm:coninggeneric}
Let $d \geq 1$ and let $G = (V,E)$ be a graph. 
\begin{enumerate}[label=(\alph*)]
\item \cite{coningwhiteley} $G$ is rigid in $\RR^d$ ($\Rd$-independent, respectively) if and only if the cone graph $G^v$ is rigid in $\RR^{d+1}$ ($\mathcal{R}_{d+1}$-independent, respectively).
\item \cite{coning} $G$ is globally rigid in $\RR^d$ if and only if the cone graph $G^v$ is globally rigid in $\RR^{d+1}$.
\end{enumerate}
\end{theorem}

\subsection{Affine maps and conics at infinity}

We say that a framework $(G,p)$ in $\CC^d$ has \emph{full affine span} if the affine span of the image of the vertices under $p$ is all of $\CC^d$. A configuration $q \in \CC^{nd}$, viewed as a point $q = (q_v)_{v \in V}$, is an \emph{affine image} of $p$ if $q_v = Ap_v + b, v \in V$ for some matrix $A \in \CC^{d \times d}$ and vector $b \in \CC^d$. We say that $p$ and $q$ are \emph{strongly congruent} if $q$ can be obtained as the affine image of $p$ under a rigid motion, i.e.\ an affine map $x \mapsto Ax + b$, such that $A^TA$ is the identity matrix.

Two frameworks $(G,p)$ and $(G,q)$ in $\RR^d$ are strongly congruent if and only if they are congruent. This is not always the case for frameworks in $\CC^d$. However, congruent frameworks that have full affine span are strongly congruent, see \cite[Corollary 8]{GT}.

Let $(G,p)$ be a framework in $\CC^d$. We say that the \emph{edge directions of $(G,p)$ lie on a conic at infinity} if there is a non-zero symmetric matrix $Q$ such that for every edge $uv$ of $G$, $(p(u) - p(v))^T Q (p(u) - p(v)) = 0$. The following lemma is implied by results of Connelly (see e.g.\ \cite[Proposition 4.2]{connelly}). For completeness, we give a proof.

\begin{lemma}\label{lemma:connelly}
Let $G = (V,E)$ be a graph and  $(G,p)$ a framework in $\CC^d$ 
such that its edge directions do not lie on a conic at infinity. Let $(G,q)$ be a framework such that $q$ is an affine image of $p$. 
Then $m_{d,G}(q) = m_{d,G}(p)$ if and only if $q$ and $p$ are congruent.
\end{lemma}
\begin{proof}
The ``if'' direction is immediate. In the other direction, suppose that $m_{d,G}(q) = m_{d,G}(p)$. Let 
$x \mapsto Ax + b$ be an affine transformation that sends $p$ to $q$. It follows from the definitions that for any pair of vertices $u,v \in V$,
\begin{equation*}
    m_{uv}(q) = (p(u) - p(v))^T A^T A (p(u) - p(v)).
\end{equation*}
%where the definition of $m_{uv}$ is extended to all (possibly non-adjacent) vertex pairs $u,v$ in a natural way. 
Therefore we have 
\begin{equation}\label{eq:conic}
     m_{uv}(q) -  m_{uv}(p) = (p(u) - p(v))^T (A^T A - I) (p(u) - p(v)).
\end{equation}
By assumption, for every edge $uv \in E$, the left-hand side of (\ref{eq:conic}) is zero. Since the edge directions of $(G,p)$ do not lie on a conic at infinity, this implies $A^TA - I= 0$, so that the left-hand side is zero for every pair of vertices $u,v \in V$, which is what we wanted to show. 
\end{proof}

The following lemma is stated in \cite{connelly} for frameworks in $\RR^d$, but the same proof works for frameworks in $\CC^d$.
\begin{lemma}\cite[Proposition 4.3]{connelly}\label{lemma:mindegreeconic}
Let $G$ be a graph in which each vertex has degree at least $d$. Then for every generic realization $(G,p)$ in $\CC^d$, the edge directions of $(G,p)$ do not lie on a conic at infinity.
\end{lemma}

The following lemma is folklore.

\begin{lemma}\label{lemma:affineimagestress}
Let $(G,p), (G,q)$ be frameworks in $\CC^d$ and suppose that $q$ is an affine image of $p$. Then $S(G,p) \subseteq S(G,q)$. If both $(G,p)$ and $(G,q)$ have full affine span, then $S(G,p) = S(G,q)$.
\end{lemma}
\begin{proof}
Let $x \mapsto Ax + b$ be the affine transformation that maps $p$ to $q$ and let $A'$ be the $nd \times nd$ block matrix with $n$ copies of $A$ in its diagonal and zeroes elsewhere, where $n$ denotes the number of vertices of $G$. In other words, $A'$ is the Kronecker product $I_n \otimes A$ of the $n \times n$ identity matrix and $A$. 

Direct calculation shows that $R(G,q) = R(G,p)A'$, which immediately implies $S(G,p) = \ker(R(G,p)^T) \subseteq \ker(R(G,q)^T) = S(G,q)$. If $(G,p)$ and $(G,q)$ have full affine span, then the affine map sending $p$ to $q$ must necessarily be invertible, so that $p$ is an affine image of $q$ as well, implying $S(G,q) \subseteq S(G,p)$.
\end{proof}

Finally, we shall use the following property of globally rigid graphs which is easy to deduce from
previous results on global rigidity and maximum rank stress matrices. We sketch the proof and
refer the reader to \cite{connelly,GHT} for the definitions and key theorems.

\begin{theorem}\label{theorem:globallyrigidcharacterization}
Let $G$ be a globally rigid graph on $n\geq d+2$ vertices in $\RR^d$, for some $d\geq 1$ and $(G,p)$ a 
 generic realization 
of $G$ in  $\CC^d$.
For  every  realization $(G,q)$ in $\CC^d$
with $S(G,p) = S(G,q)$ 
we must have  that $q$ is an affine image of $p$. 
\end{theorem}

\begin{proof} Let $(G,p_0)$ be a generic realization of $G$ in $\RR^d$. %(which is
%also generic in $\CC^d$).
It was shown in \cite{GHT} that there exists an equilibrium stress $\omega_0$ for $(G,p_0)$ for which the associated stress matrix has rank $n-d-1$. 
The complex version of \cite[Lemma 5.8]{GHT} then implies that $(G,p)$ has an equilibrium stress $\omega$ such that the associated stress matrix has rank $n - d - 1$.

If $S(G,p) = S(G,q)$ for some realization $(G,q)$, then $\omega$ is a stress for $(G,q)$ as well. Then (the complex version of) \cite[Proposition 1.2]{connelly}
implies that $q$ is an affine image of $p$.
\end{proof}

\subsection{Algebraic geometry background}

We briefly recall the notions from algebraic geometry that we shall use. For a more detailed exposition, see \cite[Appendix A]{GTT} or \cite[Section 2.2]{GJ}.
We say that a subset $X \subseteq \CC^m$ is a \emph{variety} if it is the set of simultaneous vanishing points of some polynomials $f_1,\ldots,f_k \in \CC[x_1,\ldots,x_m]$. The varieties in $\CC^m$ form the closed sets of the so-called \emph{Zariski topology}. For an arbitrary set $X \subseteq \CC^m$, we shall use $\overline{X}$ to denote its closure in the Zariski topology on $\CC^m$.

Let $X \subseteq \CC^m$ be a variety. We denote by $I(X) \subseteq \CC[x_1,\ldots,x_m]$ the set of polynomials that vanish on $X$. We say that $X$ is \emph{irreducible} if it cannot be written as the proper union of a finite number of varieties, and it is \emph{defined over $\QQ$} if $I(X)$ has a generating set consisting of polynomials with rational coefficients. The \emph{dimension} of an irreducible variety $X$ is the largest number $k$ such that there exists a chain $X_0 \subsetneq X_1 \subsetneq \cdots \subsetneq X_k = X$ of irreducible varieties. From this definition the following useful fact is immediate: if $X,Y \subseteq \CC^m$ are irreducible varieties of the same dimension with $X \subseteq Y$, then $X = Y$. We shall also use the following result.
\begin{theorem}\label{theorem:productirreducible}\cite[Chapter 3, Theorem 1.6 and Chapter 6, Example 1.33]{shafarevich}
Let $X \subseteq \CC^{m_1},Y \subseteq \CC^{m_2}$ irreducible varieties. Then the cartesian product $X \times Y \subseteq \CC^{m_1 + m_2}$ is an irreducible variety of dimension $\dim(X) + \dim(Y)$.
\end{theorem}

Let $X \subseteq \CC^m$ be an irreducible variety. At each point $x \in X$, we define the \emph{Zariski tangent space} of $X$ at $x$, denoted by $T_x X$, to be the kernel of the Jacobian matrix of a set of generating polynomials of $I(X)$, evaluated at $x$. Thus, $T_x X$ is a linear subspace of $\CC^m$. We say that $x$ is \emph{smooth} if $\dim(T_x X) = \dim(X)$. If $X$ is \emph{homogeneous} (i.e.\ it can be defined by homogenous polynomials, or equivalently, $tx \in X$ for every $x \in X$ and $t \in \CC$) and $x \in X$ is a smooth point, we define the \emph{Gauss fiber corresponding to $x$} to be the set $\{y \in X: y \text{ is smooth and } T_y X = T_x X\}$.\footnote{In other words, the Gauss fiber corresponding to $x$ is the fiber over $T_x X$ of the rational map $X \dashrightarrow Gr(\dim(X),\CC^m)$, defined by the mapping $x \mapsto T_x X$ on the smooth locus of $X$. Here, $Gr(\dim(X),\CC^E)$ denotes the Grassmannian variety of $\dim(X)$-dimensional linear subspaces of $\CC^n$.}
% If X \rightarrow Y

Let $X \subseteq \CC^m$ be a variety defined over $\QQ$. We say that a point $x \in X$ is \emph{generic in $X$} if the only polynomials with rational coefficients satisfied by $x$ are those in $I(X)$. Note that a framework $(G,p)$ in $\CC^d$ is generic if and only if $p$ is generic as a point of the variety $\CC^{nd}$. If $X$ is an irreducible variety defined over $\QQ$, then every generic point of $X$ is smooth. We shall also need the following result.

\begin{lemma}\cite[Lemma A.6]{GTT}\label{lemma:genericdense}
Let $X \subseteq Y$ be irreducible varieties, with $Y$ defined over $\QQ$. Suppose that $X$ has at least one point
which is generic in $Y$. Then the points in $X$ which are generic in $Y$ are
Zariski-dense in $X$.
\end{lemma}

\subsection{The measurement variety}

Recall that for a graph $G = (V,E)$,
we denote its $d$-dimensional edge measurement map by $m_{d,G} : \CC^{nd} \rightarrow \CC^E$.

\begin{definition}
The \emph{$d$-dimensional measurement variety} of a graph $G$ (on $n$ vertices), denoted by $M_{d,G}$, is the Zariski-closure of $m_{d,G}(\CC^{nd})$.
\end{definition}

We shall frequently use the following lemma on generic points. It follows by applying \cite[Lemmas 4.4, A.7, A.8]{GTT} to the varieties $\CC^{nd}$, $M_{d,G}$ and the map $m_{d,G}$.

\begin{lemma}\label{lemma:genericimage}
Let $x\in M_{d,G}$ be a point in the measurement variety of $G$. Then $x$ is generic in $M_{d,G}$ if and only if there is a generic point $p\in \CC^{nd}$ for which
$x=m_{d,G}(p)$.
%Let $X \subseteq \CC^m,Y \subseteq \CC^{m'}$ be irreducible varieties defined over $\QQ$ and $f : \CC^m \rightarrow \CC^{m'}$ a polynomial map with rational coefficients such that $f(X) \subseteq Y$ and $f(X)$ is dense in $Y$. Then for every generic point $x \in X$, $f(x)$ is generic in $Y$, and conversely, for every generic point $y \in Y$ there is a generic point $x \in X$ such that $f(x) = y$.
\end{lemma}

It is known that the measurement variety, being the Zariski-closure of the image of an irreducible variety defined over $\QQ$, is also an irreducible variety defined over $\QQ$. It follows from the definition and basic topological considerations that if $E' \subseteq E$ is a subset of edges inducing a subgraph $G'$ of $G$, then $M_{d,G'} = \overline{\pi_{E'}(M_{d,G})}$, where $\pi_{E'} : \CC^E \rightarrow \CC^{E'}$ is the projection onto the coordinate axes corresponding to $E'$, see \cite[Lemma 3.8]{GJ}.

In what follows we shall frequently compare the measurement
varieties of different graphs, say $G = (V,E)$ and $H = (V',E')$, that have the same number of edges. Since $M_{d,G}$ and $M_{d,H}$ lie in different ambient spaces, to compare them we must specify an identification between $\CC^E$ and $\CC^{E'}$. To this end, we introduce the following notation. Let $\psi: E \rightarrow E'$ be a bijection between the edge sets of $G$ and $H$. We write that $M_{d,G} =_\psi M_{d,H}$ if $\Psi(M_{d,G}) = M_{d,H}$, where $\Psi : \CC^{E} \rightarrow \CC^{E'}$ is the mapping induced by $\psi$ in the natural way. Similarly, we write $M_{d,G} \subseteq_\psi M_{d,H}$ if $\Psi(M_{d,G}) \subseteq M_{d,H}$. 

The following results show that $\rigiditymatroid{G}$ is ``encoded'' in the measurement variety in some sense. This has been 
observed before, see e.g. \cite{GJ,GTT,RST}.
%independently observed in \cite{GJ} 
%and \cite{algmatroids1}, see also \cite{algmatroids2}. 
Using the terminology of the latter paper, the situation can be summarized by saying that the algebraic matroid corresponding to the variety $M_{d,G}$ is $\rigiditymatroid{G}$. 

\iffalse
One of the fundamental properties of an irreducible variety $V$ is its
\emph{dimension}.
This number, denoted by $\dim(V)$,
is the largest integer $k$ for which there
exists a chain 
\begin{equation*}
V=V_k \supsetneq V_{k-1} \supsetneq ... \supsetneq V_0 \neq
\varnothing 
\end{equation*}
 of irreducible subvarieties.
In the case of the measurement variety, it turns out that this dimension is
exactly the rank of the rigidity matroid of $G$. 
\fi

\begin{lemma}\label{lemma:varietydimension}
Let $G$ be a graph on $n$ vertices. Then 
\begin{equation*}
\dim(M_{d,G})=r_{d}(G).
\end{equation*}
In particular, for $n \geq d+1$ we have $\dim(M_{d,G}) \leq nd - \binom{d+1}{2}$ and equality holds if and only if $G$ is rigid in $\RR^d$. Moreover, $G$ is $\Rd$-independent if and only if $M_{d,G} = \CC^E$.
\end{lemma}

\begin{theorem}\label{thm:matroidisomorphism}
Let $G$ and $H$ be graphs with the same number of edges and suppose that $M_{d,G} =_\psi M_{d,H}$ under some edge bijection $\psi: E(G) \rightarrow E(H)$. Then $\varpsi$ defines an isomorphism between $\rigiditymatroid{G}$ and $\rigiditymatroid{H}$.
\end{theorem}

The following result shows that the measurement variety 
%(or more precisely, its embedding in $\CC^E$) 
also encodes the space of stresses of generic frameworks. This follows from the fact that $S(G,p)$ is the orthogonal complement of $\Span(R(G,p))$ in $\CC^E$ using standard results in differential geometry, see \cite[Lemma 2.21]{GHT} or \cite[Lemma 4.10]{GTT}. 
\begin{lemma}\label{lemma:spaceofstresses}
Let $G$ be a graph and $(G,p)$ a generic realization in $\CC^d$ for some $d \geq 1$. Let $x = m_{d,G}(p) \in M_{d,G}$. Then the space of stresses $S(G,p)$ is the orthogonal complement of the tangent space $T_x(M_{d,G})$ in $\CC^E$.
\end{lemma}

The lemma implies that if $(G,p)$ and $(G,q)$ are generic frameworks in $\CC^d$, then $S(G,p)\not= S(G,q)$ if and only if the Gauss fibers
corresponding to $m_{d,G}(p)$ and $m_{d,G}(q)$ are different. We shall use this corollary later.

\subsection{Unlabeled reconstruction}\label{subsec:unlabeled}

In what follows it will be convenient to use the following notions.
We say that two frameworks $(G,\p)$ and $(H,\q)$ in $\CC^d$ are {\it length-equivalent}
({\it under the bijection $\psi$})
if there is a bijection
$\psi$ between the edge sets of $G$ and $H$ such that for every edge $e$
of $G$, the complex squared length of $e$ in  $(G,\p)$ is equal to the complex squared length of
$\psi(e)$ in $(H,\q)$. For an edge bijection $\psi: E(G) \rightarrow E(H)$ and a graph isomorphism $\varphi: V(G) \rightarrow V(H)$, we say that \emph{$\psi$ is induced by $\varphi$} if for every edge $e = uv \in E(G)$ we have $\psi(e) = \varphi(u)\varphi(v)$.

\begin{definition}\label{def:strongrec}
Let $(G,\p)$ be a generic realization of the graph $G$ in $\CC^d$.
We say that $(G,\p)$ is \textit{strongly reconstructible} 
if for every generic framework $(H,\q)$ in $\CC^d$ that is
length-equivalent to $(G,\p)$ under some edge bijection $\psi: E(G) \rightarrow E(H)$, where $H$ has the same number of vertices as $G$,
$\psi$ is induced by a graph isomorphism $\varphi : V(G) \rightarrow V(H)$.
\end{definition}

In this paper we shall mainly focus on the following stronger property, where
the condition on the number of vertices of $H$ is omitted.

\begin{definition}\label{def:fullrec}
Let $G$ be a graph without isolated vertices and let $(G,\p)$ be a generic realization of $G$ in $\CC^d$.
We say that $(G,\p)$ is \textit{fully reconstructible} 
if for every generic framework $(H,\q)$ in $\CC^d$ that is
length-equivalent to $(G,\p)$ under some edge bijection $\psi: E(G) \rightarrow E(H)$, where $H$ has no isolated vertices,
$\psi$ is induced by a graph isomorphism $\varphi : V(G) \rightarrow V(H)$. 
\end{definition}

Note that, since we assume $(G,\p)$ to be generic, its edge lengths are pairwise distinct,
and hence the bijection $\psi$ is unique in the above definitions. We also point out that when considering full reconstructibility, it is natural to only consider graphs without isolated vertices, since from any framework we can create other length-equivalent frameworks by adding isolated vertices. On the other hand, in the case of strong reconstructibility it is sensible to consider graphs with isolated vertices. In fact, it follows immediately from the definitions that $G$ is fully reconstructible in $\CC^d$ if and only if every graph obtained from $G$ by adding zero or more isolated vertices is strongly reconstructible in $\CC^d$.

%Also note that in the definition of full reconstructibility it is essential to only consider {\em generic} length-equivalent frameworks $(H,q)$, since for any $(G,p)$ and any forest $H$ with $|E|$ edges, we can find a (not necessarily generic) realization $(H,q)$ that is length-equivalent to $(G,p)$.
As Theorem \ref{theorem:strongmeasurereconstructibility} below shows, both strong and full reconstructibility of a generic framework can be characterized in terms of a certain uniqueness condition on the measurement variety $M_{d,G}$ of the underlying graph, and in fact it is this formulation of reconstructibility that we shall most commonly use throughout the paper. This also implies that these reconstructibility notions are generic properties of a graph in the sense that if there is a generic framework $(G,p)$ in $\CC^d$ which is strongly  (resp.\ fully) reconstructible, then every generic realization of $G$ in $\CC^d$ is strongly (resp.\ fully) reconstructible. This motivates the following definition.

\begin{definition}\label{def:graphrec}
A graph $G$ is said to be \textit{(generically) strongly reconstructible}
(respectively \textit{(generically) fully reconstructible}) in 
$\CC^d$ if every generic realization $(G,p)$ of $G$ in $\CC^d$
is strongly (respectively fully) reconstructible.
\end{definition}

\begin{theorem}\label{theorem:strongmeasurereconstructibility}
Let $G$ be a graph and $d \geq 1$ be fixed. The following are equivalent.
\begin{enumerate}[label=(\alph*)]
    \item $G$ is generically strongly reconstructible (generically fully reconstructible, respectively) in $\CC^d$.
    \item There exists some generic framework $(G,\p)$ in $\CC^d$ that is strongly reconstructible (fully reconstructible, respectively).
 %   \item $M_{d,G}$ fully determines $G$ and whenever $M_{d,G}$ is invariant under a permutation $\psi$ of the edges of $G$, $\psi$ is induced by a graph automorphism.
    \item Whenever $M_{d,G} =_\psi M_{d,H}$ under an edge bijection $\psi: E(G) \rightarrow E(H)$ for some graph $H$, where $H$ has the same number of vertices as $G$ (where $H$ has an arbitrary number of vertices, respectively), $\psi$ is induced by a graph isomorphism.
\end{enumerate}
\end{theorem}

The ``strongly reconstructible'' part of Theorem \ref{theorem:strongmeasurereconstructibility} is \cite[Theorem 3.4]{GJ}. The same proof works for the ``fully reconstructible'' version after omitting the condition on the number of vertices of $H$.

We close this section by recalling the main result of \cite{GTT}.

\begin{theorem}\label{thm:theranstronglyrec} \cite[Theorem 3.4]{GTT}
Let $G$ be a graph on at least $d+2$ vertices, where $d \geq 1$. Suppose that 
\begin{itemize}
    \item $d = 1$ and $G$ is $3$-connected, or
    \item $d \geq 2$ and $G$ is globally rigid in $\RR^d$.
\end{itemize}
Then $G$ is strongly reconstructible in $\CC^d$.
\end{theorem}

In the next section, we shall strengthen this result by proving that globally rigid graphs on at least $d+2$ vertices are, in fact, fully reconstructible in $\CC^d$ for $d\geq 2$. 
The cases $d = 1,2$ were already settled in \cite{GJ} by verifying the following equivalence.

\begin{theorem}\label{theorem:GJmain}\cite[Theorem 5.19, Corollary 5.22, Theorem 5.1]{GJ}
Let $G$ be a graph on at least $d+2$ vertices and without isolated vertices, where $d \in \{1,2\}$. Then the following are equivalent.
\begin{itemize}
    \item $d = 1$ and $G$ is $3$-connected or $d=2$ and $G$ is globally rigid in $\RR^2$.
    \item $G$ is strongly reconstructible in $\CC^d$.
    \item $G$ is fully reconstructible in $\CC^d$.
\end{itemize}
\end{theorem}

In Section \ref{section:examples}, we shall give examples showing that for $d \geq 3$, there are fully reconstructible graphs in $\CC^d$ (on at least $d+2$ vertices) that are not globally rigid in $\RR^d$. On the other hand, we do not know any example of a strongly reconstructible graph in $\CC^d$ on at least $d+2$ vertices that is not fully reconstructible in $\CC^d$, although it seems likely that such a graph exists.

\section{Necessary conditions for global rigidity}\label{section:main}

In this section we prove our main results: globally rigid graphs in $\RR^d$ on at least $d+2$ vertices are $\Rd$-connected (Theorem \ref{theorem:mconnected}) and fully reconstructible in $\CC^d$ (Theorem \ref{theorem:fullyreconstructible}). We start with some technical results about the structure of the measurement variety that we shall use in these proofs. Apart from Lemma \ref{lemma:productvariety}, the lemmas in the next subsection are implicit in \cite{GTT}. 

\subsection{The structure of the measurement variety}

The next lemma implies that the measurement variety of an $\Rd$-separable graph $G$ is the product of the
measurement varieties of its $\Rd$-connected components. A special case of this statement when $G$ contains an $\Rd$-bridge was proved in \cite[Theorem 3.13]{GJ}.

\begin{lemma}\label{lemma:productvariety}
Let $d \geq 1$ and let $G = (V,E)$ be a graph. Suppose that there is a partition $E = E_1 \cup E_2$  of $E$ into non-empty subsets such that $r_d(E) =r_d(E_1) + r_d(E_2)$.
%, where $\rk$ denotes the rank function of the $d$-dimensional rigidity matroid. 
Let $G_1$ and $G_2$ be the subgraphs induced by $E_1$ and $E_2$, respectively. Then $M_{d,G} = M_{d,G_1} \times M_{d,G_2}$ (under the identification $\CC^E = \CC^{E_1} \times \CC^{E_2})$.
\end{lemma}
\begin{proof}
For $i=1,2$, $M_{d,G_i}$ arises as the Zariski-closure of the projection of $M_{d,G}$ onto the coordinate axes corresponding to $E_i$, so we have that $M_{d,G} \subseteq M_{d,G_1} \times M_{d,G_2}$. By Theorem \ref{theorem:productirreducible} and Lemma \ref{lemma:varietydimension}, $M_{d,G_1} \times M_{d,G_2}$ is an irreducible variety of dimension $r_d(E_1) + r_d(E_2)$. Since this dimension equals the dimension $r_d(E)$ of the irreducible variety $M_{d,G}$, the two varieties must be equal.
\end{proof}

\begin{lemma}\label{lemma:affines}
Let $G$ be a graph and $(G,p)$ a generic framework in $\CC^d$ with full affine span. Let $\affineclass \subseteq \CC^{nd}$ denote the set of affine images of $p$, and let $\gaussfiber \subseteq M_{d,G}$ be the Gauss fiber corresponding to $m_{d,G}(p)$. Then $\overline{m_{d,G}(\affineclass)} \subseteq \overline{F}$.  
\end{lemma}
\begin{proof}
For clarity, we shall write $m$ and $M$ instead of $m_{d,G}$ and $M_{d,G}$ in the following. For any $q \in \affineclass$, if $q$ has full affine span, then by Lemma \ref{lemma:affineimagestress} we have $S(G,p) = S(G,q)$.

%As noted above, for any generic framework $(G,q)$ we have $S(G,q) = \ann(T_{m(q)} M)$. 
Let $\affineclass^g$ denote the set of frameworks in $\affineclass$ that are generic. Since $(G,p)$ is generic, this set is non-empty, and since $\affineclass$ is irreducible (being a linear space), Lemma \ref{lemma:genericdense} implies $\overline{\affineclass^g} = \affineclass$. Now for any $q \in \affineclass^g$, we have $S(G,q) = S(G,p)$ by Lemma \ref{lemma:affineimagestress}, and it follows by Lemma \ref{lemma:spaceofstresses} that $T_{m(q)} M = T_{m(p)} M$, or in other words, $m(q) \in F$.

This shows that $m(\affineclass^g) \subseteq F$. Taking Zariski-closures and using the continuity of $m$ with respect to the Zariski topology, we have 
\begin{equation*}
    \overline{F} \supseteq \overline{m(\affineclass^g)} = \overline{m(\overline{\affineclass^g})} = \overline{m(\affineclass)},
\end{equation*}
as desired.
\end{proof}

Although we shall not use this fact, we note that by \cite[Lemma 4.6]{GTT}, $m_{d,G}(\affineclass)$ is a linear space, and in particular it is closed.%, so in the above lemma we could have equivalently written $m_{d,G}(\affineclass)$ instead of $\overline{m_{d,G}(\affineclass)}$.

\iffalse
\begin{lemma}\label{lemma:affinedim}
Suppose that the generic framework $(G,p)$ in $\CC^d$ has full affine span and that its edge directions do not lie on a conic at infinity. Let $\affineclass \subseteq \CC^{nd}$ denote the set of affine images of $p$. Then $\dim(\overline{m_{d,G}(\affineclass)}) = \binom{d+1}{2}$.
\end{lemma}
\begin{proof}
Consider the function $m = m_{d,G}$ as a dominant map from $\affineclass$ to $V = \overline{m(\affineclass)}$. Since $\affineclass$ is irreducible, so is $V$. 
%and moreover $m(p)$ is a generic point of $V$. 
Thus by Theorem \ref{theorem:fiber} we have, for an open set of points $q \in \affineclass$,

\begin{equation*}
    \dim(V) = \dim(\affineclass) - \dim(m^{-1}(m(q))).
\end{equation*}

Let us choose such a $q$ with full affine span. It is easy to see that $\dim(\affineclass) = d^2 + d$ and since the edge directions of $(G,p)$ do not lie on a conic at infinity, Lemma
\ref{lemma:connelly} implies that $m^{-1}(m(q))$ is isomorphic to the variety of Euclidean motions of $\CC^d$ and so has dimension $\binom{d}{2} + d$. This shows that $\dim(V) = d^2 - \binom{d}{2} = \binom{d+1}{2}$. 
\end{proof}
\fi

Let $G$ be a graph and $d \geq 1$. We say that a Gauss fiber $F$ of $M_{d,G}$ is \emph{generic} if it contains a point that is generic in $M_{d,G}$. It will also be convenient to use the following notion in the next proof. Let $d \geq 2$ and let $n$ denote the number of vertices of $G$. We say that a framework $(G,p)$ in $\CC^d$ is a \emph{lifting} of the framework $(G,q)$ in $\CC^{d-1}$ if $(G,q)$ is obtained by projecting the image of each vertex in $(G,p)$ onto the first $d-1$ coordinate axes. If $(G,q)$ is generic, then the generic liftings of $(G,q)$ form a dense subset of the space of liftings of $(G,q)$. This follows from the basic fact that for any finite set $S \subseteq \CC$ that is algebraically independent over $\QQ$, the numbers $x \in \CC$ for which $S \cup \{x\}$ is also algebraically independent form a dense subset of $\CC$.

\begin{lemma}\label{lemma:infinitegaussfiber}
Let $G = (V,E)$ be a graph and $d \geq 2$.
\begin{enumerate}[label=\emph{\alph*)}]
    \item If $G$ is not $\Rd$-independent, then for every point $x \in M_{d-1,G} \subseteq M_{d,G}$ that is generic in $M_{d-1,G}$ there are an infinite number of generic Gauss fibers $F$ of $M_{d,G}$ with $x \in \overline{F}$.
    \item If $G$ is globally rigid in $\RR^d$ on at least $d+2$ vertices and $x \in M_{d,G} \setminus M_{d-1,G}$, then there are at most a finite number of generic Gauss fibers $F$ of $M_{d,G}$ with $x \in \overline{F}$. 
\end{enumerate}
\end{lemma}
\begin{proof}
\emph{a)} [Following \cite[Proposition 4.21]{GTT}] For clarity, we shall write $m$ instead of $m_{d,G}$ in the following. We note first that since $G$ is not $\Rd$-independent, it has at least $d+2$ vertices and consequently any generic realization of $G$ in $\CC^d$ has full affine span. 

Let $x \in M_{d-1,G}$ be a generic point. By Lemma \ref{lemma:genericimage} there is a generic framework $(G,q)$ in $\CC^{d-1}$ with $m_{d-1,G}(q) = x$. It is enough to find an infinite sequence of generic frameworks $(G,p_i), i \in \mathbb{N}$ in $\CC^d$, with corresponding (generic) Gauss fibers $F_i, i \in \mathbb{N}$, such that ${F}_i \neq {F}_j$ for $i \neq j$ and such that $q$ is an affine image of $p_i$, since by Lemma \ref{lemma:affines} this implies $x \in \overline{F}_i$.

We shall find such frameworks $(G,p_i)$ inductively. In fact, each framework will be a lifting of $(G,q)$. For the base case, let $p_1$ be an arbitrary generic lifting of $(G,q)$. Now suppose that for some $i > 1$, we have already found suitable frameworks $(G,p_j), j < i$. Since $G$ is not $\Rd$-independent, each of these frameworks has a non-zero stress $\omega_j$. For a given lifting $(G,p)$ of $(G,q)$ to have $\omega_j$ as an equlibrium stress, the last coordinates $p(v)_d, v \in V$ must satisfy $|V|$ linear equations determined by $\omega_j$, and since $\omega_j$ is non-zero, some of these equations are non-trivial. It follows that for each $j < i$, the liftings of $(G,q)$ that do not have $\omega_j$ as an equilibrium stress form a dense open subset of the space of liftings of $(G,q)$. This implies that we can find a generic lifting $(G,p_i)$ that does not satisfy any of the stresses $\omega_j, j < i$, and so in particular $S(G,p_i) \neq S(G,p_j)$ for $j < i$.
Let $F_i$ denote the Gauss fiber corresponding to $m(p_i)$. 
%Recalling that $S(p_i) = \ann(T_{m(p_i)} M_{d,G})$ 
By Lemma \ref{lemma:spaceofstresses} we must have $T_{m(p_i)} M_{d,G} \neq T_{m(p_j)} M_{d,G}$, and hence $F_i \neq F_j$ for $j < i$.

\vspace{0.3em}
\emph{b)} This is an immediate consequence of Proposition 4.20 and Remark 4.8 of \cite{GTT}.
\end{proof}

\begin{corollary}\label{corollary:measurementvarietycontainment}
Let $G$ be a globally rigid graph in $\mathbb{R}^d$ on at least $d+2$ vertices for some $d \geq 2$ and suppose that $M_{d,G} =_\varpsi M_{d,H}$ under some edge bijection $\varpsi$ for some graph $H$ not necessarily on the same number of vertices as $G$. Then $M_{d-1,H} \subseteq_\psi M_{d-1,G}$.
\end{corollary}
\begin{proof}
For clarity, we shall suppress the role of $\varpsi$ and assume that $M_{d,G}$ and $M_{d,H}$ are in the same ambient space $\CC^E$, so that $M_{d,G} = M_{d,H}$. Let $R$ be the ring of polynomials over $\CC$ with variables indexed by $E$ and let $I(M_{d-1,G}),I(M_{d-1,H}) \subseteq R$ be the set of polynomials vanishing on $M_{d-1,G}$ and $M_{d-1,H}$, respectively. Moreover, let $R_\QQ,I_\QQ(M_{d-1,G})$ and $I_\QQ(M_{d-1,H})$ denote the subset of $R,I(M_{d-1,G})$ and $I(M_{d-1,H})$, respectively, consisting of polynomials with rational coefficients.

Theorem \ref{theorem:hendrickson} implies that $G$ has no $\Rd$-bridges, so in particular it cannot be $\Rd$-independent. By Theorem \ref{thm:matroidisomorphism} it follows that $H$ is not $\Rd$-independent either, and thus
part \emph{a)} of Lemma \ref{lemma:infinitegaussfiber} implies that there is a generic point $x \in M_{d-1,H} \subseteq M_{d,H}$ such that there is an infinite number of 
generic Gauss fibers $F \subseteq M_{d,H}$ with $x \in \overline{F}$. Part \emph{b)} of the same lemma then implies that $x$ is in $M_{d-1,G}$. From this and the fact that $x$ is generic in $M_{d-1,H}$ we have the following chain of containments:
\begin{equation}\label{eq:ideal}
I_\QQ(M_{d-1,G}) \subseteq \{f \in R_{\QQ}: f(x) = 0\} = I_\QQ(M_{d-1,H}).
\end{equation}

Since both $M_{d-1,G}$ and $M_{d-1,H}$ are defined over $\QQ$, $I_\QQ(M_{d-1,G})$ and $I_\QQ(M_{d-1,H})$ generate $I(M_{d-1,G})$ and $I(M_{d-1,H})$, respectively. Thus (\ref{eq:ideal}) also implies $I(M_{d-1,G}) \subseteq I(M_{d-1,H})$, which is equivalent to $M_{d-1,H} \subseteq M_{d-1,G}$. 
\end{proof}

\subsection{Globally rigid graphs are \texorpdfstring{$\Rd$}{Rd}-connected}

We are ready to prove the first main result of this section.
%if $G$ is globally rigid in $\RR^d$ on $n\geq d+2$ vertices then ${\cal R}_d(G)$ is connected. 

\begin{theorem}
\label{theorem:mconnected}
Let $G = (V,E)$ be a globally rigid graph in $\RR^d$ on $n\geq d+2$ vertices. 
Then $G$ is $\Rd$-connected.
\end{theorem}

\begin{proof} 

Let $(G,p)$ be a generic framework in $\CC^d$.
We shall show that if $G$ is not $\Rd$-connected, then 
there is a framework $(G,q)$ in $\CC^d$ such that $S(G,p) = S(G,q)$ but 
where $q$ is not an affine image of $p$.
Then, from Theorem \ref{theorem:globallyrigidcharacterization},
$G$ cannot be globally rigid in $\RR^d$, and we are done.

If $G$ is not $\Rd$-connected, 
there must be a partition $E = E_1 \cup E_2$ of $E$ into non-empty subsets with $r_d(E) = r_d(E_1) + r_d(E_2)$. Let $G_1 = (V,E_1)$ and $G_2 = (V,E_2)$ and let $(G_1,p), (G_2,p)$ denote the respective sub-frameworks of $(G,p)$.  

%If $G$ has an $\Rd$-bridge $e$, then we may assume that $E_2=\{e\}$ and $G_1=G-e$. By Theorem \ref{theorem:hendrickson},
%$G$ is $(d+1)$-connected, so in this case the minimum degree of $G_1$ is at least $d$.
%Thus, by Lemma \ref{lemma:mindegreeconic}, the edge directions of $(G_1,p)$ do not lie on a conic at infinity.

By Theorem \ref{theorem:hendrickson}, $G$ contains no $\Rd$-bridges, that is, every edge is contained in an $\Rd$-circuit. Since any $\Rd$-circuit that contains edges from $E_1$ must be contained in $G_1$, we have that $G_1$ contains an $\Rd$-circuit, so in particular it has a subgraph of minimum degree at least $d + 1$. 
Thus, by Lemma \ref{lemma:mindegreeconic}, applied to this subgraph, the edge directions of $(G_1,p)$ do not lie on a conic at infinity.

Now we find our promised $(G,q)$.
By Lemma \ref{lemma:productvariety}, $M_{d,G} = M_{d,G_1} \times M_{d,G_2}$. Let $m_{d,G}(p) = (x_1,x_2) \in M_{d,G}$. Now $x_2 \in M_{d,G_2}$ implies $4x_2 \in M_{d,G_2}$ and consequently $(x_1,4x_2) \in M_{d,G}$. Since $(G,p)$ was generic, $(x_1,x_2)$ is generic in $M_{d,G}$ by Lemma \ref{lemma:genericimage}, and this implies that $(x_1,4x_2)$ is generic in $M_{d,G}$ as well. Using Lemma \ref{lemma:genericimage}, it follows that there is a generic framework $(G,q)$ in $\CC^d$ with $m_{d,G}(q) = (x_1,4x_2)$. 

 Let us consider the sub-frameworks 
 $(G_1,q)$ and $(G_2,q)$. 
 Both of these frameworks are generic. Also note that $(G_1,q)$ is equivalent to  $(G_1,p)$ and $(G_2,q)$ is equivalent to $(G_2,2p)$. 
 Since $S(G,p) = S(G,2p)$, this implies $S(G,p) = S(G_1,p) \oplus S(G_2,p) = S(G_1,q) \oplus S(G_2,q) = S(G,q)$, as desired. 

We have that $(G_1,q)$ is equivalent to $(G_1,p)$.
As established above,  the edge directions of $(G_1,p)$ do not lie on a conic at infinity. 
Thus it follows 
from Lemma \ref{lemma:connelly} that if $q$ is an affine image of $p$, then $q$ must be
congruent to $p$.
But $q$ is clearly not congruent to $p$ as $(G_2,p)$ is not equivalent to $(G_2,q)$.
Thus $p$ is not an affine image of $q$, as desired.
\end{proof}

Theorem \ref{theorem:mconnected} was known to hold in $\RR^1$ (where global rigidity, $2$-connectivity and $\mathcal{R}_1$-connectivity are equivalent) and in $\RR^2$, see \cite{JJconnrig}
(c.f. Theorem \ref{theorem:JacksonJordan} above). We note that in \cite{JT} it is conjectured that globally rigid graphs are ``non-degenerate'' in $\RR^d$, a condition that is stronger than $\Rd$-connectivity.
%It is conjectured in \cite{JT} that globally rigid graphs in $\RR^d$ are ``non-degenerate'' in $\RR^d$ (for the definition see \cite{JT}). Since non-degenerate graphs are $\Rd$-connected (\cite[Lemma 3.2]{JT}), Theorem \ref{theorem:mconnected} gives an affirmative answer to a weaker, $\Rd$-connected version of this conjecture. 

% some words about second proof
Underlying the proof of Theorem \ref{theorem:mconnected} is the following structural observation on the measurement variety of $G$, which we give without details. Since $G$ is globally rigid, \cite[Theorem 4.4]{GHT} and \cite[Lemma 4.24, Remark 4.25]{GTT} imply that for any generic Gauss fiber $F$ of $M_{d,G}$ we have $\dim(\overline{F}) = \binom{d+1}{2}$. On the other hand, it follows from the definitions that if $M_{d,G} = M_{d,G_1} \times M_{d,G_2}$, then $F = F_1 \times F_2$ where $F_1$ and $F_2$ are some generic Gauss fibers of $M_{d,G_1}$ and $M_{d,G_2}$, respectively. Using Lemma \ref{lemma:affines} it can be shown that under the assumptions on $G_1$ made in the proof of Theorem \ref{theorem:mconnected}, we have $\dim(\overline{F_1}) \geq \binom{d+1}{2}$. Since $\dim(\overline{F_2}) \geq 1$, this gives 
\begin{equation*}
\binom{d+1}{2} = \dim(\overline{F}) = \dim(\overline{F_1}) + \dim(\overline{F_2}) > \binom{d+1}{2},
\end{equation*}
a contradiction.

\subsection{Globally rigid graphs are fully reconstructible}

Our goal in this subsection is to prove the following result, which gives an affirmative answer to \cite[Question 7.5]{GTT}.

\begin{theorem}\label{theorem:fullyreconstructible}
Let $d \geq 2$ and let $G$ be a graph on $n \geq d+2$ vertices that is globally rigid in $\RR^d$. Then $G$ is fully reconstructible in $\CC^d$.
\end{theorem}
Theorem~\ref{thm:gluingstronglyreconstructible},
below, will lead to examples
that show that global rigidity
is not necessary for full reconstructibility.

%Theorem \ref{theorem:fullyreconstructible} gives an affirmative answer to \cite[Question 7.5]{GTT}.
%The $d = 2$ case follows from Theorem 5.1 and Corollary 5.22 of \cite{GJ}. In fact, the cited results show that a graph on at least four vertices and without isolated vertices is fully reconstructible in $\CC^2$ if and only if it is globally rigid in $\RR^2$. Results from the same paper also imply that a graph on at least four vertices and without isolated vertices is fully reconstructible in $\CC^1$ if and only if it is $3$-connected. In Section \ref{section:examples}, we shall give examples showing that in $d \geq 3$, there are fully reconstructible graphs (on at least $d+2$ vertices) that are not globally rigid.

Our proof of Theorem \ref{theorem:fullyreconstructible} uses Theorem \ref{theorem:mconnected}. In fact, as the following theorem shows, the former result is a strengthening of the latter.

\begin{theorem}\label{thm:fullyrecmconnected}
Let $G = (V,E)$ be a graph without isolated vertices and suppose that $G$ is fully reconstructible in $\CC^d$. Then $G$ is $\Rd$-connected.
\end{theorem}
\begin{proof}
Suppose, for a contradiction, that there is a partition $E = E_1 \cup E_2$ of $E$ into non-empty subsets with $r_d(E) = r_d(E_1) + r_d(E_2)$, and let $G_1$ and $G_2$ be the subgraphs induced by $E_1$ and $E_2$, respectively. By Lemma \ref{lemma:productvariety}, this implies $M_{d,G} = M_{d,G_1} \times M_{d,G_2}$. If $G_1$ and $G_2$ have at least one vertex in common, then let $H$ be the 
graph consisting of disjoint copies of $G_1$ and $G_2$.
%disjoint union of $G_1$ and $G_2$. 
Otherwise, let $H$ be the graph obtained from disjoint copies of $G_1$ and $G_2$ by identifying some vertex of $G_1$ and some vertex of $G_2$.
%
%by gluing some vertex of $G_1$ with some vertex of $G_2$ in $G$. 
In both cases, we have $m_{d,H}(\CC^{n'd}) = m_{d,G_1}(\CC^{n_1d}) \times m_{d,G_2}(\CC^{n_2d})$, where $n'$, $n_1$ and $n_2$ denote the number of vertices of $H$, $G_1$ and $G_2$, respectively. It follows that $M_{d,H} = M_{d,G_1} \times M_{d,G_2}$.\footnote{This follows from the basic fact that for any $U \subseteq \CC^{E_1}, V \subseteq \CC^{E_2}$ we have $\overline{U \times V} = \overline{U} \times \overline{V}$, where closures are meant in the respective Zariski topologies, see e.g.\ \cite[Lemma 2.4]{GJ}.} Since by construction $G$ and $H$ are not isomorphic, Theorem \ref{theorem:strongmeasurereconstructibility} implies that $G$ is not fully reconstructible, as desired.
\end{proof}
This theorem is similar in spirit to ~\cite[Theorem 5.21]{GJ}, which states that
if $G$ is strongly reconstructible in $\CC^d$ on at least $d+2$ vertices and without isolated vertices, then it cannot contain
an $\Rd$-bridge.
Example \ref{ex:6ring} in the next section shows that the converse of Theorem \ref{thm:fullyrecmconnected} is not true: the $6$-ring depicted in Figure \ref{figure:6ring} is not even strongly reconstructible in $\CC^3$, despite being $\mathcal{R}_3$-connected, $4$-connected and redundantly rigid in $\RR^3$.

We now turn our attention to Theorem \ref{theorem:fullyreconstructible}. As our argument is quite involved, we start by giving a brief outline of the proof. By Theorem \ref{theorem:strongmeasurereconstructibility}, to show that the globally rigid graph $G$ is fully reconstructible, we need to show that whenever $M_{d,G} =_\psi M_{d,H}$ for some graph $H$ without isolated vertices and some edge bijection $\psi: E(G) \rightarrow E(H)$, we have that $H$ is isomorphic to $G$ (and the isomorphism induces the appropriate edge bijection). Let $n$ and $n'$ denote the number of vertices of $G$ and $H$, respectively. If $n = n'$, then we are done by the strong reconstructibility of $G$ (Theorem \ref{thm:theranstronglyrec}). 
If $n' < n$, then $M_{d,H}$ must necessarily be of lower dimension than
$M_{d,G}$, which is impossible.
The only remaining possibility to be ruled out is that 
$n'>n$; note that in this case $M_{d,G} =_\psi M_{d,H}$ (and in particular the equality of dimensions) implies that $H$ is locally flexible in $\RR^d$. 
%but the dimension of $M_{d,H}$ is not higher than that 
%of $M_{d,G}$ due to local flexibility (and thus infinitesimal flexibility) 
%of $H$!

We rule this out as follows.
 From $M_{d,G} =_\psi M_{d,H}$ and $M_{d-1,H} \subseteq_\psi M_{d-1,G}$ (which follows from Corollary \ref{corollary:measurementvarietycontainment}) we get bounds on $k_d$ and $k_{d-1}$, the generic dimension of the space of infinitesimal motions of $H$ in $d$ and $d-1$ dimensions, respectively. Using Theorem \ref{theorem:mconnected}, we also get that $H$ is $\Rd$-connected, and in particular connected. The crux of our argument is Corollary \ref{corollary:dofbound}, which states that for a connected graph, $k_d - \binom{d+1}{2}$ (the ``$d$-dimensional degrees of freedom'' of the graph) is larger by a multiplicative factor than $k_{d-1} - \binom{d}{2}$, the $(d-1)$-dimensional degrees of freedom of the graph. Applying this to $H$, we shall get a contradiction with the bounds on $k_d$ and $k_{d-1}$ obtained previously.
 
\iffalse
 The crux of our argument is Lemma \ref{lemma:infinitesimalmotionbound}, which states that as we decrease $d$, this generic dimension cannot decrease ``too much''. This leads to a lower bound on the number $k_1$ of one-dimensional infinitesimal motions of $H$. On the other hand, using Theorem \ref{theorem:mconnected}, we can deduce that $H$ is $\Rd$-connected, and in particular connected, so that $k_1 = 1$, contradicting the lower bound we obtained previously. 
\fi

Before proving Corollary \ref{corollary:dofbound}, we need a number of technical results. The first is about the structure of a certain variety, while the second gives a concrete example of an infinitesimal motion in $d$ dimensions that, generically, cannot be decomposed into two $(d-1)$-dimensional infinitesimal motions. The subsequent lemmas use this result to obtain bounds on the generic dimension of infinitesimal motions in $\CC^i$ for all $1 \leq i \leq d$.

%The crux of our argument is a technical result (Lemma \ref{lemma:infinitesimalrotation}) which describes a particular infinitesimal motion in $d$ dimensions that, generically, cannot be decomposed into two $(d-1)$-dimensional infinitesimal motions in a particular way. 

\begin{lemma}\label{lemma:hyperplanecontainment}
Let $G = (V,E)$ be a graph on $n$ vertices.
\begin{enumerate}[label=\emph{\alph*)}]
\item The variety $M_{1,G} \subseteq \CC^E$ is not contained in any (linear) hyperplane in $\CC^E$.

\item  Consider the mapping $f : \CC^{2n} \rightarrow \CC^E$ defined by 
\begin{equation*}
p = \left(p_v^{(1)},p_v^{(2)}\right)_{v \in V} \longmapsto \left((p^{(1)}_v - p^{(1)}_u)(p^{(2)}_v - p^{(2)}_u)\right)_{uv \in E}.
\end{equation*}
Then $\overline{f(\CC^{2n})}$ (and consequently $f(\CC^{2n})$ itself) is not contained in  any (linear) hyperplane  in $\CC^E$.
\end{enumerate}
\end{lemma}
\begin{proof}
\emph{a)} Let $H$ be an arbitrary hyperplane in $\CC^E$ whose orthogonal complement is generated by some non-zero $\omega \in \CC^E$. The configurations $p \in \CC^{n}$ for which $\omega$ is an equilibrium stress of $(G,p)$ form a proper linear subspace of $\CC^{n}$. In particular, we can find a generic framework $(G,p)$ in $\CC^1$ which does not have $\omega$ as an equilibrium stress. By Lemma \ref{lemma:spaceofstresses}, $\omega$ is not in the orthogonal complement of the tangent space of $M_{1,G}$ at $m_{1,G}(p)$. It follows that this tangent space is not contained in $H$, which implies that $M_{1,G}$ is not contained in $H$ either, as desired.  

\vspace{0.3em}
\emph{b)}
By direct calculation\footnote{This was already observed in \cite{lamannumber}.} we have that $f = m_{2,G} \circ \alpha$ where $\alpha : \CC^{2n} \rightarrow \CC^{2n}$ is defined by

\begin{equation*}
    p = \left(p^{(1)}_v,p^{(2)}_v\right)_{v \in V} \longmapsto \left(\frac{p^{(1)}_v + p^{(2)}_v}{2}, \frac{p^{(1)}_v - p^{(2)}_v}{2\sqrt{-1}} \right)_{v \in V}.
\end{equation*}

Since $\alpha$ is a linear automorphism of $\CC^{2n}$, this implies that $f(\CC^{2n}) = m_{2,G}(\CC^{2n})$, and thus $\overline{f(\CC^{2n})} = M_{2,G}$. Now by part \emph{a)}, $M_{1,G}$ is not contained in any linear hyperplane in $\CC^E$. Since $M_{1,G} \subseteq M_{2,G}$, the same holds for $M_{2,G}$. 
\end{proof}

For the next two lemmas, we introduce the following notation. Let $d \geq 3$ and let $G = (V,E)$ be a graph on $n$ vertices. For a framework $(G,p)$ in $\CC^d$, let $W_1(G,p) \leq \CC^{nd}$ denote the set of infinitesimal motions of $(G,p)$ that are supported on the first $d-1$ coordinates, that is, infinitesimal motions of the form $\left(q_v,0\right)_{v \in V}$, where $q_v \in \CC^{d-1}$ for each vertex $v$. Similarly, let $W_2(G,p)$ denote the set of infinitesimal motions that are supported on the last $d-1$ coordinates, and let $W(G,p)$ be the subspace of $\CC^{nd}$ spanned by $W_1(G,p) \cup W_2(G,p)$.

\begin{lemma}\label{lemma:infinitesimalrotation}
Let $d \geq 3$ and let $G = (V,E)$ be a graph on $n$ vertices that is not $\mathcal{R}_{d-2}$-independent. Then for any generic configuration 
$p = \left(p_v^{(1)},\dots,p^{(d)}_v\right)_{v \in V} \in \CC^{nd}$, the infinitesimal rotation $\varphi = \varphi(p)$ of $(G,p)$ defined by
\begin{equation*}
    \varphi_v = (- p_v^{(d)}, 0, \ldots, 0, p^{(1)}_v) \hspace{2em} \forall v \in V
\end{equation*}
is not in the subspace $W(G,p)$.
%of $\CC^{nd}$ spanned by $W_1(G,p) \cup W_2(G,p)$.
In particular, $W(G,p)$ is a proper subspace of the space of infinitesimal motions of $(G,p)$.
\end{lemma}
\begin{proof}
Consider an arbitrary realization $(G,p)$ in $\CC^d$ and let $\widetilde{p} \in \CC^{n(d-2)}$ denote the projection of $p$ onto the middle $d-2$ coordinate axes. Suppose that $\varphi(p)$ can be written as $\varphi(p) = q + r$, where $q \in W_1(G,p)$ and  $r \in W_2(G,p)$. Then $q_v = (- p_v^{(d)}, \widetilde{q}_v, 0)$ for every vertex $v \in V$, where $\widetilde{q} = (\widetilde{q}_v)_{v \in V} \in \CC^{n(d-2)}$. Since $q$ is an infinitesimal motion, we have that
\[\left(p_v - p_u\right) \cdot \left(q_v - q_u\right) = 0, \hspace{2em} \forall uv \in E.\]
Using the above description of $q$ and rearranging gives  
%for every edge $uv \in E$, we have
\begin{equation}\label{eq:1}
     (\widetilde{p}_v - \widetilde{p}_u)\cdot (\widetilde{q}_v - \widetilde{q}_u) = (p_v^{(1)} - p_u^{(1)})(p_v^{(d)} - p_u^{(d)}), \hspace{2em} \forall uv \in E.
\end{equation}

We shall show that for a generic realization $(G,p)$ of $G$ in $\CC^d$, (\ref{eq:1}) cannot hold for any vector $\left(\widetilde{q}_v\right)_{v \in V}$. We start by observing that $(G,\widetilde{p})$ is a generic framework in $\CC^{d-2}$, so $\rank(R(G,\widetilde{p}))$ is equal to the rank $r = r_{d-2}(G)$ of $G$, which is the maximal rank of the rigidity matrix of any framework in $\CC^{d-2}$. Note that 
\[\big((\widetilde{p}_v - \widetilde{p}_u)\cdot (\widetilde{q}_v - \widetilde{q}_u)\big)_{uv \in E} = R(G,\widetilde{p})\widetilde{q},\]
so it is sufficient to show that the vector $z = \left((p_v^{(1)} - p_u^{(1)})(p_v^{(d)} - p_u^{(d)}\right)_{uv \in E}$ is not contained in $\Span(R(G,\widetilde{p}))$. This is the same as saying $\rk(M(p)) = r + 1$, where $M(p)$ is obtained by appending the column vector $z$ to $R(G,\widetilde{p})$, which is further equivalent to the non-vanishing of some $(r+1) \times (r+1)$ subdeterminant of $M(p)$. The entries of $M(p)$ are polynomial functions in the coordinates of $(G,p)$ with rational coefficients, and consequently so are these $(r+1) \times (r+1)$ subdeterminants. 

Since $(G,p)$ is generic, it is sufficient to show that at least one of the polynomial functions describing these subdeterminants is not identically zero; in other words, that there is some framework $(G,p')$ for which the corresponding matrix $M(p')$ has rank $r+1$. In fact, we can find such a framework by only modifying the first and last coordinates of $p_v$ for each $v \in V$. Since $G$ is not $\mathcal{R}_{d-2}$-independent, $r < |E|$ and thus $\Span(R(G,\widetilde{p}))$ is contained in some linear hyperplane $H \subseteq \CC^E$. By part \emph{b)} of Lemma \ref{lemma:hyperplanecontainment}, we can find some vectors $(x_v)_{v \in V}$ and $(y_v)_{v \in V}$ such that the vector $\left((x_v - x_u)(y_v-y_u)\right)_{uv \in E}$ is not contained in $H$. It follows that the framework $(G,p')$ defined by $p'_v = (x_v,\widetilde{p}_v,q_v), v \in V$ satisfies $\rk(M(p')) = r+1$, as desired.
\end{proof}

\begin{lemma}\label{lemma:infinitesimalmotionbound}
Let $d \geq 3$ and let $G$ be a graph on $n$ vertices that is not $\mathcal{R}_{d-2}$-independent. For $i = 1,\ldots,d$, let $k_i$ denote the dimension of the space of infinitesimal motions of a generic realization of $G$ in $\CC^i$. Then for $i = 1,\ldots,d-2$, we have 
\[k_{i+2} - k_{i+1} \geq k_{i+1} - k_{i} + 1.\]
\end{lemma}
\begin{proof}
The assumption that $G$ is not $\mathcal{R}_{d-2}$-independent implies that it is not $\mathcal{R}_i$-independent for all $1 \leq i \leq d-2$. Thus, it is sufficient to prove for $i = d-2$. By Lemma \ref{lemma:infinitesimalrotation}, there is a generic realization $(G,p)$ in $\CC^d$ such that $W(G,p)$ 
%\subseteq \CC^{nd}$ generated by the subspaces $W_1(G,p)$ and $W_2(G,p)$ 
is a proper subset of the space of infinitesimal motions of $(G,p)$. Note that $\dim(W_1(G,p)) = \dim(W_2(G,p)) = k_{d-1}$, and by the choice of $(G,p)$, we have $\dim(W(G,p)) \leq k_d - 1$. Moreover, the subspace $W' = W_1(G,p) \cap W_2(G,p)$ consists of the infinitesimal motions of $(G,p)$ that are supported on the middle $d-2$ coordinates. This implies $\dim(W') = k_{d-2}$. By basic linear algebra we have
\begin{equation*}
    \dim(W') + \dim(W(G,p)) = \dim(W_1(G,p)) + \dim(W_2(G,p)). 
\end{equation*}
Substituting the above equalities and inequality and then rearranging gives
\begin{equation*}
    k_{d} - k_{d-1} \geq k_{d-1} - k_{d-2} + 1,
\end{equation*}
as desired.
\end{proof}

If $G$ is $\mathcal{R}_{d-2}$-independent, then the conclusion of Lemma \ref{lemma:infinitesimalmotionbound} does not hold. In this case $G$ is $\mathcal{R}_{d-1}$-independent and $\Rd$-independent as well, so we have $k_i = ni - |E|$ for $d - 2 \leq i \leq d$, so that $k_{d} - k_{d-1} = k_{d-1} - k_{d-2}$.

The following combinatorial lemma lets us turn the recursive bound on $k_i$ given in Lemma \ref{lemma:infinitesimalmotionbound} into a lower bound that only depends on $k_d$ and $k_{d-1}$.

\begin{lemma}\label{lemma:inductivebound}
Let $d \geq 2$ be an integer and let $k_1,k_2,\ldots,k_d \in \ZZ$ be a sequence of integers with $k_d = \binom{d+1}{2} + x$ and $k_{d-1} = \binom{d}{2} + y$ for some $x,y \in \ZZ$. Suppose that for $1 \leq i \leq d-2$ we have $k_{i+2} - k_{i+1} \geq k_{i+1} - k_{i} + 1$. Then for $1 \leq i \leq d-1$ we have $k_i \geq \binom{i+1}{2} + (d-i)(y-x) + x$.
\end{lemma}
\begin{proof}
Consider the numbers $l_i = k_i - \binom{i+1}{2}$ for $i = 1,\ldots,d$. Our goal is to prove $l_i \geq (d-i)(y-x) + x$. By definition, we have $l_{d-1} - l_{d} = y - x$. We claim that $l_i - l_{i+1} \geq l_{i+1} - l_{i+2}$ holds for all $i = 1, \ldots, d-2$. Indeed, using the bound on $k_i$ we obtain
\begin{align*}
    l_i = k_i - \binom{i+1}{2} &\geq 2k_{i+1} - k_{i+2} + 1 - \binom{i+1}{2} \\ &= (2 \binom{i+2}{2} - \binom{i+3}{2} - \binom{i+1}{2} + 1) + (2l_{i+1} - l_{i+2}) \\
    &=2l_{i+1} - l_{i+2},
\end{align*}
where we used the fact that for all $a \geq 1$, $\binom{a}{2} + \binom{a+2}{2} = 2\binom{a+1}{2} + 1$.
This also gives
\begin{align*}
    l_i - l_{i+1} \geq l_{i+1} - l_{i+2} \geq l_{i+2} - l_{i+3} \geq \cdots \geq l_{d-1} - l_{d} = y-x.
\end{align*}

The statement now follows easily by induction on $j = d-i$. For $j=1$, we need $l_{d-1} \geq y$, which is actually satisfied with equality. Now let us assume $1 < j < d$. Then by induction we have 
\begin{align*}
    l_i \geq l_{i+1} + y - x \geq (d - i - 1)(y-x) + x + y - x = (d-i)(y-x) + x,
\end{align*}
as desired.
\end{proof}

If $G$ is a connected graph, then its one-dimensional generic realizations have a one-dimensional space of infinitesimal motions. For such graphs, combining Lemma \ref{lemma:infinitesimalmotionbound} with the bound on $k_1$ given by Lemma \ref{lemma:inductivebound} gives the following corollary.

\begin{corollary}\label{corollary:dofbound}
Let $d \geq 3$ be an integer and $G$ a connected graph, and suppose that $G$ is not $\mathcal{R}_{d-2}$-independent. Let $k_{d-1}$ and $k_d$ denote the dimension of the space of infinitesimal motions of a generic realization of $G$ in $\CC^{d-1}$ and in $\CC^d$, respectively. Finally, let $\dof_{d-1}(G) = k_{d-1} - \binom{d}{2}$ and $\dof_{d}(G) = k_d - \binom{d+1}{2}$ denote the ``degrees of freedom'' of $G$ in $d-1$ and $d$ dimensions, respectively. Then we have
\[\dof_d(G) \geq \frac{d-1}{d-2}\dof_{d-1}(G).\]
\end{corollary}

Now we are ready to prove our main theorem.

\begin{proof}[Proof of Theorem \ref{theorem:fullyreconstructible}]
The $d=2$ case follows from Theorem \ref{theorem:GJmain}, so we only prove for $d \geq 3$. Let $H$ be a graph on $n'$ vertices, without isolated vertices and such that $M_{d,G} =_\psi M_{d,H}$ under some edge bijection $\varpsi$. Then by Corollary \ref{corollary:measurementvarietycontainment}, we also have $M_{d-1,H} \subseteq_\psi M_{d-1,G}$. Observe that since $G$ is globally rigid, it is $\Rd$-connected by Theorem \ref{theorem:mconnected}, and thus so is $H$ by Theorem \ref{thm:matroidisomorphism}
%Lemma \ref{lemma:productvariety} 
and the equality of the measurement varieties. In particular, $H$ is a connected graph and it is not $\Rd$-independent
and thus not $\mathcal{R}_{d-2}$-independent.

Let $s = n'-n$ and let $k_i$ denote the dimension of the space of infinitesimal motions of a generic realization of $H$ in $\CC^i$ for $1 = 1,\ldots,d$. By considering the dimension of $M_{d,G}$ (using Lemma \ref{lemma:varietydimension}) and using the assumption that $G$ is (globally) rigid, we get $n'd - k_d = nd - \binom{d+1}{2}$, implying $k_d = sd + \binom{d+1}{2}$. Similarly, from $M_{d-1,H} \subseteq_\psi M_{d-1,G}$ we have $n'(d-1) - k_{d-1} \leq n(d-1) - \binom{d}{2},$ so that $k_{d-1} \geq s(d-1) + \binom{d}{2}$. Note that here we used the fact that 
if $G$ is (globally) rigid in $\RR^d$, then it is rigid in $\RR^{d-1}$, see e.g.\ \cite[Theorem 63.2.11]{JW}.

Since $k_d \geq \binom{d+1}{2}$, we must have $s \geq 0$. 
%This establishes that $n' \geq n$. 
On the other hand, since $H$ is connected and not $\mathcal{R}_{d-2}$-independent, Corollary \ref{corollary:dofbound} implies that
\[sd = \dof_d(H) \geq \frac{d-1}{d-2}\dof_{d-1}(H) \geq \frac{d-1}{d-2}s(d-1).\]
Reordering, we get
\[sd(d-2) \geq s(d-1)^2,\] which is equivalent to $s \leq 0$.
It follows that $s = 0$, so $G$ and $H$ have the same number of vertices. By Theorem \ref{thm:theranstronglyrec}, $G$ is strongly reconstructible in $\CC^d$, so $\psi$ is induced by a graph isomorphism $\varphi : G \rightarrow H$, as desired.

\iffalse
On the other hand, from  Lemmas \ref{lemma:infinitesimalmotionbound} and \ref{lemma:inductivebound} we get (substituting $x = sd$ and $y \geq s(d-1)$ into Lemma \ref{lemma:inductivebound})

\begin{equation*}
k_i \geq \binom{i+1}{2} - (d-i)s + sd = \binom{i+1}{2} + si 
\end{equation*}
for $1 \leq i \leq d-2$, and in particular, $k_1 \geq s + 1$. But since $H$ is connected, we know that generically the only $1$-dimensional infinitesimal motions of $H$ are translations, so that $k_1 = 1$. It follows that $s = 0$, so $G$ and $H$ have the same number of vertices. By Theorem \ref{thm:theranstronglyrec}, $G$ is strongly reconstructible in $\CC^d$, so $\psi$ is induced by a graph isomorphism $\varphi : G \rightarrow H$, as desired.
\fi
\end{proof}

%For $d=1$ a graph $G$ is fully reconstructible if and only if it is $3$-connected.
%We can also deduce that for $d=2$ a graph $G$ is fully reconstructible if and only if it is a
%globally rigid graph on at least four vertices. The characterization of fully
%reconstructible graphs in $\CC^3$ remains open.

Applying Theorem \ref{theorem:fullyreconstructible} to a globally rigid subgraph of a graph, we obtain the following corollary.

\begin{corollary}\label{corollary:globallyrigidsubgraph}
Let $d \geq 2$ and let $(G,p)$ and $(H,q)$ be generic frameworks in $\CC^d$ that are length-equivalent under the edge bijection $\psi$. Let $G_0 = (V_0,E_0)$ be a globally rigid subgraph of $G = (V,E)$ and let $H_0$ denote the subgraph of $H$ induced by $\psi(E_0)$. Then $\left.\psi\right|_{E_0}$ is induced by an isomorphism $\varphi : V(G_0) \rightarrow V(H_0)$ and the frameworks $(G_0,\left.p\right|_{V_0})$ and $(H_0,\left.q\right|_{V(H_0)} \circ \varphi)$ are congruent.
\end{corollary}

The $d = 2$ case of Corollary \ref{corollary:globallyrigidsubgraph} can be found in \cite[Corollary 5.2]{GJ}. 
%In the cited result, the same statement is claimed to be true when $d = 1$, but this is incorrect; instead, in the one-dimensional case we need to replace ``globally rigid subgraph'' with ``$3$-connected subgraph'' to get a true statement.

To close this section, let us briefly return to Lemma \ref{lemma:infinitesimalmotionbound}.
%, which was the key step in our proof of Theorem \ref{theorem:fullyreconstructible}
Let $G$ be a graph and, as before, let $k_d$ denote the dimension of the set of infinitesimal motions of a generic realization of $G$ in $\CC^d$. Using the fact that for a graph $G$ on at least $d+2$ vertices we have $r_d(G) = nd - k_d$, we can rephrase the $i = d-2$ case of Lemma \ref{lemma:infinitesimalmotionbound} in the following way.
\begin{corollary}\label{corollary:dimensionbound}
Let $d \geq 3$ and let $G$ be a graph that is not $\mathcal{R}_{d-2}$-independent. Then the rank of $G$ satisfies
\begin{equation}\label{eq:dimension}
    r_{d}(G) - r_{d-1}(G) \leq r_{d-1}(G) - r_{d-2}(G) - 1.
\end{equation}
\end{corollary}
This bound is the best possible in the sense that if $G$ is rigid in $\RR^{d}$, then (\ref{eq:dimension}) is satisfied with equality. We note that Corollary \ref{corollary:dimensionbound} can be interpreted as a statement about the so-called secant defect of $M_{1,G}$, similar to Zak's theorem on superadditivity \cite{fantechi,zak}; see \cite[Section 4.3]{GHT} for a related discussion.

\section{Examples and open questions}\label{section:examples}

In this section, we examine various examples related to $\Rd$-connected and $\Rd$-separable graphs, as well as the unlabeled reconstruction problem. 

\subsection{New examples of \texorpdfstring{$H$}{H}-graphs}

Following \cite{JKT}, we say that a graph $G$ is an {\it $H$-graph} in $\RR^d$ if it is 
$(d+1)$-connected and redundantly rigid in $\RR^d$ (i.e.\ it satisfies the
necessary conditions of Theorem \ref{theorem:hendrickson}), but it is not globally rigid in $\RR^d$.
There are no $H$-graphs for $d=1,2$ but for
$d\geq 3$ they exist and finding more examples may lead to a better understanding of
higher dimensional global rigidity. For a long time, the complete bipartite graph $K_{5,5}$ was the only 
known $H$-graph in $\RR^3$ (identified in \cite{con2}), until infinite families had been found
in \cite{JKT}.%, see also \cite{JT2}.

Theorem \ref{theorem:mconnected} can be used to give new examples of $H$-graphs which are $\mathcal{R}_3$-separable. These also demonstrate that redundant rigidity and $(d+1)$-connectivity together do not imply $\Rd$-connectivity in the $d \geq 3$ case.

\begin{example}\label{example:hgraph}
Consider the construction illustrated in Figure \ref{figure:mseparable2}(a).
It is easy to see that the graph $G$ in the figure is $4$-connected.

\begin{claim*}
$G$ is redundantly rigid and $\mathcal{R}_3$-separable.
\end{claim*}
\begin{proof}
We show that $G$ is rigid by showing that a spanning subgraph of $G$ can be reduced to $K_4$ by a sequence of
the following operations: (i) deletion of a vertex of degree at least three, (ii) deletion of
a vertex $v$ of degree four and the addition of a new edge between two neighbours of $v$, (iii)
the contraction of an edge $uv$ for which $u$ and $v$ have exactly two common neighbours.
It is well-known that the inverse operations ($0$- and $1$-extension and vertex splitting)
preserve rigidity in $\RR^3$, see e.g.\ \cite{Whlong}. Thus, since $K_4$ is rigid, it will follow that $G$ is also rigid.
  
First, delete the nine vertices of degree four from $G$ and then delete one edge from each of the
remaining copies of $K_5$. The resulting graph has $28$ vertices and $78$ edges. We shall reduce it to its
internal $K_4$ subgraph. By the symmetry of the graph we
can perform the reduction steps in groups of four in a symmetric way. First we contract four edges of the outer ring that do not
belong to the four copies of $K_5-e$, one from each ``corner''. These operations create a vertex of degree four
in each corner, so we can apply operation (ii) at each of them to obtain a graph on $20$ vertices
and $54$ edges, in which the four edges added form a four-cycle. After that we again apply operation (ii)
in three rounds, decreasing the number of vertices by four in each round. If the added edges are
chosen appropriately, we obtain a graph on $8$ vertices, consisting of the internal $K_4$, a disjoint four-cycle $C_4$, and
eight more edges that connect them (so that they span an $8$-cycle). From here we apply operations (ii), (i), (ii), and then again (i) to get the $K_4$ subgraph.

Thus $G$ is indeed rigid, that is, $r_3(G)=3|V(G)|-6=105$.
Note that $G$ is also redundantly rigid, because every edge of $G$ belongs to a $K_5$ subgraph.
To see that $G$ is $\mathcal{R}_3$-separable, first observe that if we remove one edge from each copy of $K_5$ in the outer ring $G^o$ (say, one edge incident with
each vertex of degree four) then we do not decrease its rank and
obtain a spanning subgraph of $G^o$ with $96$ edges. Thus $r_3(G^o)\leq 96$. The inner $K_5$ has rank $9$.
Hence we must have $r_3(G)=r_3(G^o)+r_3(K_5)$, showing that $G$ is indeed $\mathcal{R}_3$-separable.
\end{proof}
\vspace{-0.5em}
Since $\mathcal{R}_3$-separable graphs are not globally rigid in $\RR^3$ by Theorem \ref{theorem:mconnected}, $G$ is indeed an $H$-graph in $\RR^3$. We can obtain an infinite family of $H$-graphs in $\RR^3$ from Example \ref{example:hgraph} by replacing the inner $K_5$ in Figure \ref{figure:mseparable2} with another $4$-connected redundantly rigid graph $K'$ in $\RR^3$ on at least five vertices (as in Figure \ref{figure:mseparable2}(b), where $K'=K_6-e$).

\end{example}

\begin{figure}[ht!]
        \begin{subfigure}{.49\linewidth}
        \centering
        \begin{tikzpicture}[x = 1cm, y = 1cm, scale = 4]
            \node[draw=none,rotate=0,minimum size=5cm,regular polygon,regular polygon sides=12] (inner12) {};
            \node[draw=none,rotate=0,minimum size=7cm,regular polygon,regular polygon sides=12] (outer12) {};
            \node[draw=none,rotate=45,minimum size=3.7cm,regular polygon,regular polygon sides=4] (inner4) {};
            \node[draw=none,rotate=45,minimum size=8.2cm,regular polygon,regular polygon sides=12] (outer8) {};
            \node (o) at (-0.14,-0.11){};
            
            \foreach \x in {1,2,...,12}
            {\pgfmathtruncatemacro\y{Mod(\x,12) + 1}
            \draw [line width=\normaledge, color=edgeblack] (outer12.corner \x) -- (outer12.corner \y);
            \draw [line width=\normaledge, color=edgeblack] (inner12.corner \x) -- (inner12.corner \y);
            \draw [line width=\normaledge, color=edgeblack] (inner12.corner \x) -- (outer12.corner \y);
            \draw [line width=\normaledge, color=edgeblack] (inner12.corner \y) -- (outer12.corner \x);
            \draw [line width=\normaledge, color=edgeblack] (inner12.corner \x) -- (outer12.corner \x);}
            
            \foreach \x in {1,2,...,4}
            {\draw [line width=\normaledge, color=edgeblack] (inner4.corner \x) -- (o);
            \pgfmathtruncatemacro\y{Mod(\x,4) + 1}
            \draw [line width=\normaledge, color=edgeblack] (inner4.corner \x) -- (inner4.corner \y);
            \pgfmathtruncatemacro\y{Mod(\x + 1,4) + 1}
            \draw [line width=\normaledge, color=edgeblack] (inner4.corner \x) -- (inner4.corner \y);
            \pgfmathtruncatemacro\y{Mod(3*\x - 2,12) + 1}
            \draw [line width=\normaledge, color=edgeblack] (inner4.corner \x) -- (inner12.corner \y);
            \draw [line width=\normaledge, color=edgeblack] (inner4.corner \x) -- (outer12.corner \y);
            \pgfmathtruncatemacro\y{Mod(3*\x - 3,12) + 1}
            \draw [line width=\normaledge, color=edgeblack] (inner4.corner \x) -- (inner12.corner \y);
            \draw [line width=\normaledge, color=edgeblack] (inner4.corner \x) -- (outer12.corner \y);
            \pgfmathtruncatemacro\z{3*\x - 1}
            \pgfmathtruncatemacro\y{Mod(3*\x,12) + 1}
            \draw [line width=\normaledge, color=edgeblack] (outer8.corner \z) -- (outer12.corner \y);
            \draw [line width=\normaledge, color=edgeblack] (outer8.corner \z) -- (inner12.corner \y);
            \pgfmathtruncatemacro\y{Mod(3*\x - 1,12) + 1}
            \draw [line width=\normaledge, color=edgeblack] (outer8.corner \z) -- (outer12.corner \y);
            \draw [line width=\normaledge, color=edgeblack] (outer8.corner \z) -- (inner12.corner \y);
            \pgfmathtruncatemacro\z{3*\x - 2}
            \pgfmathtruncatemacro\y{Mod(3*\x - 2,12) + 1}
            \draw [line width=\normaledge, color=edgeblack] (outer8.corner \z) -- (outer12.corner \y);
            \draw [line width=\normaledge, color=edgeblack] (outer8.corner \z) -- (inner12.corner \y);
            \pgfmathtruncatemacro\y{Mod(3*\x - 1,12) + 1}
            \draw [line width=\normaledge, color=edgeblack] (outer8.corner \z) -- (outer12.corner \y);
            \draw [line width=\normaledge, color=edgeblack] (outer8.corner \z) -- (inner12.corner \y);}

            \foreach \x in {1,2,...,12}
            \draw [fill=vertexblack] (inner12.corner \x) circle (1pt);
        
            \foreach \x in {1,2,...,12}
            \draw [fill=vertexblack] (outer12.corner \x) circle (1pt);

            \foreach \x in {1,2,...,4}
            \draw [fill=vertexblack] (inner4.corner \x) circle (1pt);

            \foreach \x in {1,2,...,4}
            {\pgfmathtruncatemacro\z{3*\x - 1}
            \draw [fill=vertexblack] (outer8.corner \z) circle (1pt);
            \pgfmathtruncatemacro\z{3*\x - 2}
            \draw [fill=vertexblack] (outer8.corner \z) circle (1pt);}

            \draw [fill=vertexblack] (o) circle (1pt);
        \end{tikzpicture}
        \caption{}
        \end{subfigure}
        \begin{subfigure}{.49\linewidth}
        \centering
        \begin{tikzpicture}[x = 1cm, y = 1cm, scale = 4]
            \node[draw=none,rotate=0,minimum size=5cm,regular polygon,regular polygon sides=12] (inner12) {};
            \node[draw=none,rotate=0,minimum size=7cm,regular polygon,regular polygon sides=12] (outer12) {};
            \node[draw=none,rotate=45,minimum size=3.7cm,regular polygon,regular polygon sides=4] (inner4) {};
            \node[draw=none,rotate=45,minimum size=8.2cm,regular polygon,regular polygon sides=12] (outer8) {};
            \node (o) at (-0.14,-0.11){};
            \node (o2) at (0.14,-0.11){};
            
            \foreach \x in {1,2,...,12}
            {\pgfmathtruncatemacro\y{Mod(\x,12) + 1}
            \draw [line width=\normaledge, color=edgeblack] (outer12.corner \x) -- (outer12.corner \y);
            \draw [line width=\normaledge, color=edgeblack] (inner12.corner \x) -- (inner12.corner \y);
            \draw [line width=\normaledge, color=edgeblack] (inner12.corner \x) -- (outer12.corner \y);
            \draw [line width=\normaledge, color=edgeblack] (inner12.corner \y) -- (outer12.corner \x);
            \draw [line width=\normaledge, color=edgeblack] (inner12.corner \x) -- (outer12.corner \x);}
            
            \foreach \x in {1,2,...,4}
            {\draw [line width=\normaledge, color=edgeblack] (inner4.corner \x) -- (o);
            \draw [line width=\normaledge, color=edgeblack] (inner4.corner \x) -- (o2);
            \pgfmathtruncatemacro\y{Mod(\x,4) + 1}
            \draw [line width=\normaledge, color=edgeblack] (inner4.corner \x) -- (inner4.corner \y);
            \pgfmathtruncatemacro\y{Mod(\x + 1,4) + 1}
            \draw [line width=\normaledge, color=edgeblack] (inner4.corner \x) -- (inner4.corner \y);
            \pgfmathtruncatemacro\y{Mod(3*\x - 2,12) + 1}
            \draw [line width=\normaledge, color=edgeblack] (inner4.corner \x) -- (inner12.corner \y);
            \draw [line width=\normaledge, color=edgeblack] (inner4.corner \x) -- (outer12.corner \y);
            \pgfmathtruncatemacro\y{Mod(3*\x - 3,12) + 1}
            \draw [line width=\normaledge, color=edgeblack] (inner4.corner \x) -- (inner12.corner \y);
            \draw [line width=\normaledge, color=edgeblack] (inner4.corner \x) -- (outer12.corner \y);
            \pgfmathtruncatemacro\z{3*\x - 1}
            \pgfmathtruncatemacro\y{Mod(3*\x,12) + 1}
            \draw [line width=\normaledge, color=edgeblack] (outer8.corner \z) -- (outer12.corner \y);
            \draw [line width=\normaledge, color=edgeblack] (outer8.corner \z) -- (inner12.corner \y);
            \pgfmathtruncatemacro\y{Mod(3*\x - 1,12) + 1}
            \draw [line width=\normaledge, color=edgeblack] (outer8.corner \z) -- (outer12.corner \y);
            \draw [line width=\normaledge, color=edgeblack] (outer8.corner \z) -- (inner12.corner \y);
            \pgfmathtruncatemacro\z{3*\x - 2}
            \pgfmathtruncatemacro\y{Mod(3*\x - 2,12) + 1}
            \draw [line width=\normaledge, color=edgeblack] (outer8.corner \z) -- (outer12.corner \y);
            \draw [line width=\normaledge, color=edgeblack] (outer8.corner \z) -- (inner12.corner \y);
            \pgfmathtruncatemacro\y{Mod(3*\x - 1,12) + 1}
            \draw [line width=\normaledge, color=edgeblack] (outer8.corner \z) -- (outer12.corner \y);
            \draw [line width=\normaledge, color=edgeblack] (outer8.corner \z) -- (inner12.corner \y);}

            \foreach \x in {1,2,...,12}
            \draw [fill=vertexblack] (inner12.corner \x) circle (1pt);
        
            \foreach \x in {1,2,...,12}
            \draw [fill=vertexblack] (outer12.corner \x) circle (1pt);

            \foreach \x in {1,2,...,4}
            \draw [fill=vertexblack] (inner4.corner \x) circle (1pt);

            \foreach \x in {1,2,...,4}
            {\pgfmathtruncatemacro\z{3*\x - 1}
            \draw [fill=vertexblack] (outer8.corner \z) circle (1pt);
            \pgfmathtruncatemacro\z{3*\x - 2}
            \draw [fill=vertexblack] (outer8.corner \z) circle (1pt);}

            \draw [fill=vertexblack] (o) circle (1pt);
            \draw [fill=vertexblack] (o2) circle (1pt);
        \end{tikzpicture}
        \caption{}
        \end{subfigure}
    \caption{Graphs that are $4$-connected, redundantly rigid and $M$-separable in $\RR^3$. %The graph $G_1$ above is $3$-connected and satisfies $36=r_3(G_1)=r_3(G_1^o)+r_3(K_5)=27+9$, where $G_1^o$ is the outer ring of $K_5$'s and $K_5$ is the internal $K_5$. 
    The graph shown in (a) satisfies $r_3(G)= 105 = 96 + 9 =r_3(G^o)+r_3(K_5)$, where
$G^o$ is the outer ring of $K_5$'s.}
    \label{figure:mseparable2}
\end{figure}
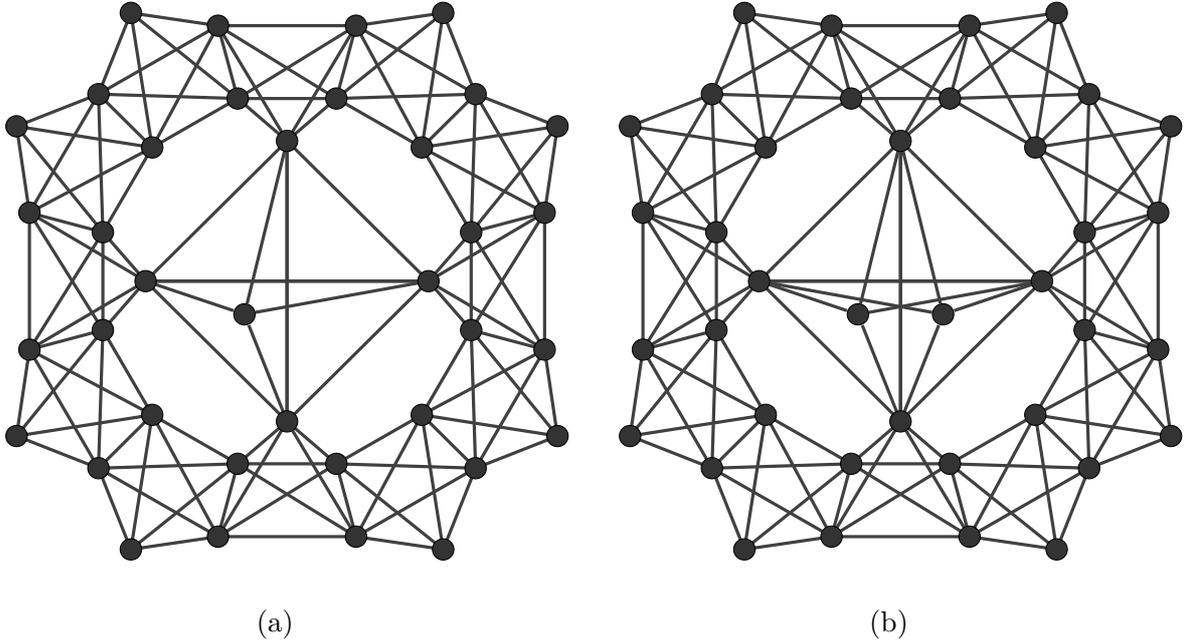

We note that some of the $H$-graphs obtained in \cite{JKT} (for example, the ``$6$-ring'' depicted in Figure \ref{figure:6ring}) show that $4$-connectivity, redundant rigidity,
and $\mathcal{R}_3$-connectivity together do not imply global rigidity in $\RR^3$.

It is also interesting to note that every known $H$-graph in $\RR^3$, except $K_{5,5}$, has a $4$-separator.

\begin{question}\label{question:K55}
Is every $5$-connected and redundantly rigid graph, other than $K_{5,5}$, globally rigid in $\RR^3$?
\end{question}

The following related question also seems to be open.

\begin{question}
Is every $5$-connected and redundantly rigid graph $\mathcal{R}_3$-connected?
\end{question}

%\begin{question}
%Is every $5$-connected and redundantly rigid graph, other than $K_{5,5}$, globally rigid in $\RR^3$?
%\end{question}

It is easy to see that coning increases vertex connectivity by one, so that the cone of a $(d+1)$-connected graph is $(d+2)$-connected. It is also folklore (and follows from Theorem \ref{thm:coninggeneric}(a) and Lemma \ref{lemma:mcircuitcone} in the next section) that if $G$ is redundantly rigid in $\RR^d$, then its cone graph $G^v$ is redundantly rigid in $\RR^{d+1}$. Together with Theorem \ref{thm:coninggeneric}(b) these imply that the cone of a $H$-graph in $\RR^d$ is a $H$-graph in $\RR^{d+1}$.   
We can use this fact to construct further families of $H$-graphs.
However, these graphs will not be $\mathcal{R}_{d+1}$-separable (see Theorem \ref{lemma:conemconnected}).
Instead, we can use higher dimensional body-hinge graphs \cite{JKT} to generalize the $\mathcal{R}_{3}$-separable
construction of Figure \ref{figure:mseparable2} to $d\geq 4$. We omit the details.

\subsection{Unlabeled reconstructibility and small separators}

In \cite[Question 7.2]{GTT} the authors asked whether every graph $G$ that is $3$-connected and redundantly rigid in $\RR^d$ is determined by its measurement variety; that is, whether  $M_{d,G} =_\psi M_{d,H}$ under some edge bijection $\psi$ implies that $G$ and $H$ are isomorphic (note that here we do not require that the isomorphism induces $\psi$). Such a graph was called ``weakly reconstructible in $\CC^d$'' in \cite{GJ}.
In the other direction, in \cite[Section 7]{GJ}, the authors asked whether every graph on at least $d+2$ vertices that is strongly reconstructible in $\CC^d$ for some $d \geq 3$ is globally rigid in $\RR^d$, or (more weakly) whether it is $(d+1)$-connected.

In this subsection we provide negative answers to each of these questions for $d\geq 3$. 
Throughout this section we shall use the following (folklore) result which also appears in the proof of \cite[Theorem 5.21]{GJ}. 

\begin{lemma}\label{lemma:sameedgemap}
Let $G$ and $H$ be graphs on at least three vertices with $G$ connected and let $\varphi_1, \varphi_2: V(G) \rightarrow V(H)$ be injective graph homomorphisms. Suppose that $\varphi_1$ and $\varphi_2$ induce the same edge map $\varpsi: E(G) \rightarrow E(H)$. Then $\varphi_1(v) = \varphi_2(v)$ for all $v \in V(G)$.
\end{lemma}
\begin{proof}
Let $v \in V(G)$ be a vertex of degree at least two and let $vu,vu' \in E(G)$ be a pair of edges incident to $v$. Now $\varphi_1(v)$ is the unique vertex in $H$ that is an end-vertex of both $\varpsi(vu)$ and $\varpsi(vu')$. Since $\varphi_2(v)$ can be described in the same way, we have $\varphi_1(v) = \varphi_2(v)$. Note that in a connected graph on at least three vertices, every edge has at least one end-vertex with degree at least two. This shows that $\varphi_1$ and $\varphi_2$ send at least one vertex of each edge in $G$ to the same vertex in $H$. Since they also send each edge in $G$ to the same edge in $H$, they must agree on every vertex of $G$.  
\end{proof}

\begin{example}\label{example:hgraphreconstructible}
Consider again the graph $G$ shown in Figure \ref{figure:mseparable2}(a). As we have seen in Example \ref{example:hgraph}, $G$ is $4$-connected, redundantly rigid and $\mathcal{R}_3$-separable.
\begin{claim*}
$G$ is not strongly reconstructible in $\CC^3$.
\end{claim*}
\begin{proof}\vspace{-0.2em}
By the $\mathcal{R}_3$-separability of $G$ and Lemma \ref{lemma:productvariety}, we have $M_{d,G} = M_{d,G^o} \times M_{d,K_5}$, where $K_5$ denotes the complete subgraph of $G$ induced by the inner five vertices. Let $\varpsi$ be a permutation of the edges of $G$ that leaves the edges of $G^o$ in place and permutes the edges of $K_5$ according to some permutation of its five vertices. Then $M_{d,G} =_\varpsi M_{d,G}$, i.e.\ $M_{d,G}$ is invariant under the permutation of the coordinate axes in $\CC^{E(G)}$ induced by $\varpsi$. However, $\varpsi$ is not induced by a graph automorphism of $G$: since it leaves the edges of $G^o$ in place, by Lemma \ref{lemma:sameedgemap} such an automorphism would have to leave the vertices of $G^o$ in place and thus be the identity map on $G$, which does not induce $\varpsi$. By Theorem \ref{theorem:strongmeasurereconstructibility}, this shows that $G$ is not strongly reconstructible.
\end{proof}
\vspace{-0.5em}
It can be shown similarly that the graph $G'$ shown in Figure \ref{figure:mseparable2}(b) is not even weakly reconstructible in $\CC^3$: the graph obtained by adding the missing edge to the inner $K_6$ and removing a different edge from it has the same measurement variety as $G'$, even though the two graphs are not isomorphic. These examples can also be generalized to higher dimensions using the results on body-hinge graphs in \cite{JKT}.
\end{example}

The graphs considered in Example \ref{example:hgraphreconstructible} are all $\mathcal{R}_3$-separable. The next example is of an $\mathcal{R}_3$-connected graph that is not strongly reconstructible in $\CC^3$.

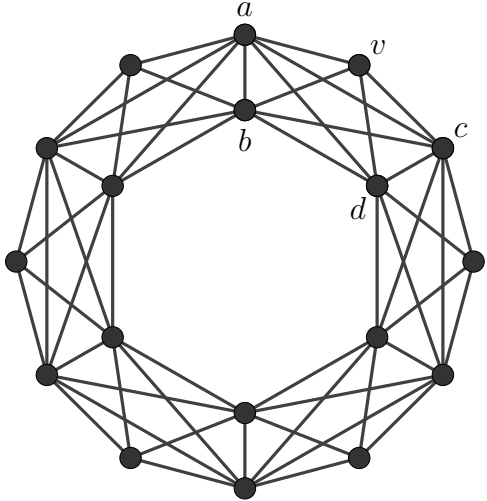
\begin{figure}[ht]
    \centering
        \begin{tikzpicture}[x = 1cm, y = 1cm, scale = 4]
            \node[draw=none,rotate=90,minimum size=4cm,regular polygon,regular polygon sides=6] (a) {};
            \node[draw=none,rotate=45,minimum size=6cm,regular polygon,regular polygon sides=12] (b) {};
            \pgfmathtruncatemacro{\n}{6}
            \pgfmathtruncatemacro{\m}{\n - 2}
            
            \foreach \x in {1,2,...,6}
            {\pgfmathtruncatemacro\y{Mod(\x,6) + 1}
            \draw [line width=\normaledge, color=edgeblack] (a.corner \x) -- (a.corner \y);
            \pgfmathtruncatemacro\z{2*\x - 1}
            \draw [line width=\normaledge, color=edgeblack] (a.corner \x) -- (b.corner \z);
            \pgfmathtruncatemacro\z{2*\x}
            \draw [line width=\normaledge, color=edgeblack] (a.corner \x) -- (b.corner \z);
            \pgfmathtruncatemacro\z{Mod(2*\x,12) + 1}
            \draw [line width=\normaledge, color=edgeblack] (a.corner \x) -- (b.corner \z);
            \pgfmathtruncatemacro\z{Mod(2*\x + 1,12) + 1}
            \draw [line width=\normaledge, color=edgeblack] (a.corner \x) -- (b.corner \z);
            \pgfmathtruncatemacro\z{Mod(2*\x - 3,12) + 1}
            \draw [line width=\normaledge, color=edgeblack] (a.corner \x) -- (b.corner \z);
            \pgfmathtruncatemacro\y{Mod(2*\x + 3,12) + 1}
            \pgfmathtruncatemacro\z{Mod(2*\x + 1,12) + 1}
            \draw [line width=\normaledge, color=edgeblack] (b.corner \y) -- (b.corner \z);
            }

            \foreach \x in {1,2,...,12}
            {\pgfmathtruncatemacro\y{Mod(\x,12) + 1}
            \draw [line width=\normaledge, color=edgeblack] (b.corner \x) -- (b.corner \y);
            }
            
            \foreach \x in {1,2,...,6}
            \draw [fill=vertexblack] (a.corner \x) circle (1pt);
            \foreach \x in {1,2,...,12}
            \draw [fill=vertexblack] (b.corner \x) circle (1pt);
            
            \draw [fill = vertexblack] (b.corner 11) circle (1pt) node[above right]{$v$};
            \draw [fill = vertexblack] (b.corner 12) circle (1pt) node[above=3pt]{$a$};
            \draw [fill = vertexblack] (b.corner 10) circle (1pt) node[above right]{$c$};
            \draw [fill = vertexblack] (a.corner 6) circle (1pt) node[below=3pt]{$b$};
            \draw [fill = vertexblack] (a.corner 5) circle (1pt) node[below left]{$d$};
        
        \end{tikzpicture}
        \caption{A graph that is $4$-connected, redundantly rigid in $\RR^3$ and $M$-connected in $\RR^3$, but not globally rigid in $\RR^3$.}
    \label{figure:6ring}
\end{figure}

\begin{example}
\label{ex:6ring}
Let $G$ be the $6$-ring of $K_5$'s shown in Figure \ref{figure:6ring}. As noted before, it is $4$-connected, redundantly rigid in $\RR^3$ and $\mathcal{R}_3$-connected.
%We shall prove that it is not strongly reconstructible in $\CC^3$.

\begin{claim*} $G$ is not strongly reconstructible in $\CC^3$.
\end{claim*}
\begin{proof}\vspace{-0.2em}
%To see this, 
Let $v$ be a vertex of degree four and let us denote the vertices of the $K_5$ subgraph
that contains $v$ by $\{v,a,b,c,d\}$, where the edges $ab$ and $cd$ are shared
by neighbouring $K_5$'s. Let $H = G - v$.  It is easy to check that the edges $ac,ad,bc,bd$ are all $\mathcal{R}_3$-bridges in $H$.
Thus by Lemma \ref{lemma:productvariety}, we have $M_{d,H}=M_{d,H'}\oplus \CC^4$, where $H'=H-\{ac,ad,bc,bd\}$. Let $(G,p)$ be a generic realization of $G$ and (by a slight abuse of notation)
let $(H,p)$ be its restriction to $H$.

Consider the permutation $\varpsi$ of the edges of $H$ that leaves the edges of $H'$ in place and maps $ac,ad,bc,bd$ to $bd,bc,ad,ac$, respectively. Since $M_{d,H}$ is invariant under the permutation of coordinate axes induced by $\varpsi$, Lemma \ref{lemma:genericimage} implies that there exists a generic realization $(H,q)$ of $H$ in $\CC^3$ such that $(H',q|_{V(H')})$ and $(H',p|_{V(H')})$ are equivalent and $m_{ac}(p)=m_{bd}(q)$, $m_{bd}(p)=m_{ac}(q)$, $m_{bc}(p)=m_{ad}(q)$, and $m_{ad}(p)=m_{bc}(q)$. This implies that the point configurations $(p(a),p(b),p(c),p(d))$ and $(q(b),q(a),q(d),q(c))$ are congruent, i.e.\ the points in them have the same pairwise squared distances. Since by genericity they have full affine span, this also implies that they are strongly congruent, in other words, there is a rigid motion of $\CC^3$ that maps $p(a), p(b), p(c), p(d)$ to $q(b),q(a),q(d),q(c)$, respectively. 

Let us extend $(H,q)$ to a realization $(G,q)$ by defining $q(v)$ to be the image of $p(v)$ under this rigid motion, and let us also extend $\varpsi$ to all of $E(G)$ by mapping $va,vb,vc,vd$ to $vb,va,vd,vc$, respectively. Then $m_{d,G}(p) = m_{d,G}(q)$ under the edge permutation $\varpsi$. On the other hand, $\varpsi$ is not induced by a graph automorphism of $G$: since it leaves the edges of $H'$ in place, by Lemma \ref{lemma:sameedgemap} such an automorphism would leave the vertices of $H'$ in place and consequently it would have to be the identity map on $G$, which does not induce $\varpsi$. By Theorem \ref{theorem:strongmeasurereconstructibility}, this shows that $G$ is not strongly reconstructible in $\CC^3$, as desired. 
%Now we can reinsert $v$ into $(H,q)$ (or equivalently, we can merge a relabeled copy of the %$K_5$ of $v$ in $(G,p)$ and
%$(H',q)$) to obtain a realization $(G,q)$ of $G$ which is length equivalent to $(G,p)$
%but the bijection between the edges is not induced by a graph isomorphism.
\end{proof}
\end{example}

Finally, we construct examples of fully reconstructible graphs with small separators.
In the next proof, we shall use the following fact. Let $G_1, G_2$ be rigid graphs in $\RR^d$ on at least $d+1$ vertices and let $G$ be obtained from $G_1$ and $G_2$ by identifying $k$ pairs of vertices. Then if $0 \leq k \leq d - 1$, we have $r_d(G) = d|V(G)| - \binom{d+1}{2} - \binom{d-k+1}{2}$, and if $k \geq d$, then $G$ is rigid. This follows e.g.\ from the ``gluing lemma'' \cite[Lemma 11.1.9]{Whlong}, or it can be seen directly by considering the infinitesimal motions of a generic realization of $G$ in $\RR^d$.

\begin{theorem}\label{thm:gluingstronglyreconstructible}
Let $G = (V,E)$ be a graph with induced subgraphs $G_1 = (V_1,E_1)$ and $G_2 = (V_2,E_2)$ for which $V_1 \cup V_2 = V$ and $V_1 \cap V_2$ induces a connected subgraph of $G$ on at least three vertices. Let $d \geq 1$. If $G_1$ and $G_2$ are fully reconstructible rigid graphs on at least $d+1$ vertices in $\CC^d$, then $G$ is fully reconstructible in $\CC^d$.
\end{theorem}
\begin{proof}
By Theorem \ref{theorem:strongmeasurereconstructibility}, it suffices to show that if for some graph $H$ we have $M_{d,G} =_\psi M_{d,H}$ under some edge bijection $\varpsi$, then $\varpsi$ is induced by a graph isomorphism $\varphi : V(G) \rightarrow V(H)$. Let $H_1, H_2$ be the subgraphs of $H$ induced by $\varpsi(E_1)$ and $\varpsi(E_2)$, respectively. Now for $i=1,2$, $M_{d,G_i} = M_{d,H_i}$, so by the full reconstructibility of $G_i$, there is a graph isomorphism $\varphi_i : V_i \rightarrow V(H_i)$ that induces $\left.\varpsi\right|_{E_i}$. Since $V_1 \cap V_2$ induces a connected subgraph of $G$, Lemma \ref{lemma:sameedgemap} applies to $\varphi_1|_{V_1 \cap V_2}$ and $\varphi_2|_{V_1 \cap V_2}$, giving $\varphi_1(v) = \varphi_2(v)$ for all $v \in V_1 \cap V_2$. 

It follows that $H$ is the union of two subgraphs $H_1,H_2$ which are isomorphic to $G_1,G_2$, respectively, and have $k\geq |V_1\cap V_2|$ vertices in common. Let $\ell = |V_1\cap V_2|\geq 3$. We first show that $k = \ell$. Note that $|V(G)|=|V_1|+|V_2|- \ell$ and $|V(H)|=|V_1|+|V_2|- k$, and hence $|V(H)|-|V(G)|=l-k$.

%It is not hard to deduce that a graph $L$ obtained as the union of two rigid graphs (each of them on at least $d+1$ vertices) with $k$ vertices in common (for some $0\leq k\leq d-1$)  satisfies $r_d(L)=d|V(L)|-\binom{d+1}{2}-\binom{d-k+1}{2}.$ It is well-known (as the glueing lemma) that if they share at least $d$ vertices then $L$ is rigid.

Let us first consider the $k\leq d-1$ case. Since by Theorem \ref{thm:matroidisomorphism} the $d$-dimensional rigidity matroids of  $G$ and $H$ are isomorphic, we obtain
%using the fact that $G_i$ is rigid, $i=1,2$, it is not hard to deduce that 
\begin{equation*}
d|V(G)| -\binom{d+1}{2}-\binom{d-\ell +1}{2} = r_d(G) =  r_d(H) = d|V(H)|-\binom{d+1}{2}-\binom{d-k+1}{2}.
\end{equation*}
This gives
\begin{equation*}
    \binom{d+1}{2} + \binom{k}{2} = dk + \binom{d-k+1}{2} = d\ell + \binom{d-\ell+1}{2} = \binom{d+1}{2} + \binom{\ell}{2},
\end{equation*}
where the second equality comes from the previous equation and the first and third equalities from direct calculation. It follows that $k=\ell$. 
An analogous argument shows that $k = \ell$ holds in the $k\geq d$ case as well.

This implies that the only vertices of $H$ that are in the image of both $\varphi_1$ and $\varphi_2$ are those in the image of $V_1 \cap V_2$, where $\varphi_1$ and $\varphi_2$ agree. Hence we can ``glue'' $\varphi_1$ and $\varphi_2$, i.e.\ the mapping $\varphi: V \rightarrow V(H)$ defined by $\varphi|_{V_i} = \varphi_i, i = 1,2$ is well-defined and is an isomorphism (in particular, it is injective). Then $\varphi$ induces $\varpsi$, as required.
\end{proof}

In fact, the proof shows that we can slightly relax the conditions in Theorem \ref{thm:gluingstronglyreconstructible} by only requiring that each connected component of the graph induced by $V_1 \cap V_2$ has at least three vertices. 

\begin{example}\label{example:gluedcompletegraph}
Let $d \geq 1$ and let $G_d$ be the graph obtained by gluing two copies of the complete graph $K_{d+2}$ along three pairs of vertices. Theorems \ref{theorem:fullyreconstructible} and \ref{thm:gluingstronglyreconstructible} imply that $G_d$ is fully reconstructible in $\CC^d$. This example shows that fully (or strongly) reconstructible graphs need not be $(d+1)$-connected in the $d \geq 3$ case, which gives a negative answer to a question posed in \cite[Section 7]{GJ}. Also note that for $d \geq 4$, $G_d$ is not even rigid in $\RR^d$. It is unclear whether there exist non-rigid fully reconstructible graphs in $\CC^3$.
\end{example}

The previous example also shows that a fully reconstructible graph in $\CC^d$ need not be globally rigid in $\RR^d$ in the $d \geq 3$ case. Another such example is given in \cite{K55}, where it is shown, using a computer-assisted proof, that the complete bipartite graph $K_{5,5}$ is fully reconstructible in $\CC^3$. A combinatorial characterization of fully reconstructible graphs in $\CC^d$ seems elusive, even in the $d = 3$ case.

\subsection{Monotonicity of unlabeled reconstructibility}

The graph $G_d$ of Example \ref{example:gluedcompletegraph} also shows that for $d \geq 4$, edge addition does not necessarily preserve strong (or full) reconstructibility in $\CC^d$. Indeed, $G_d$ is fully reconstructible in $\CC^d$, but for any pair of non-neighbouring vertices $u,v \in V(G_d)$, $uv$ is an $\Rd$-bridge in $G_d + uv$. \cite[Theorem 5.21]{GJ} states that strongly reconstructible graphs in $\CC^d$ (on at least $d+2$ vertices and without isolated vertices) do not contain $\Rd$-bridges, so that $G_d + uv$ is not strongly reconstructible in $\CC^d$. It would be interesting to see whether this phenomenon can only happen if the newly added edge is an $\Rd$-bridge.

\begin{question}\label{question:linkedpair}
Let $d \geq 1$ and let $G = (V,E)$ be a graph on at least $d+2$ vertices that is strongly reconstructible in $\CC^d$ (fully reconstructible in $\CC^d$, respectively). Is it true that 
if for some pair of vertices $u,v \in V$ we have $uv \notin E$ and $r_d(G) = r_d(G + uv)$, then $G + uv$ is strongly reconstructible in $\CC^d$ (fully reconstructible in $\CC^d$, respectively)?
\end{question}

We can prove the following weaker result. We say that a pair $\{u,v\}$ of vertices in a graph $G$ is {\it globally linked in $G$} in $\CC^d$
if for every generic framework $(G,p)$ in $\CC^d$ and every equivalent realization $(G,q)$, we have $m_{uv}(p) = m_{uv}(q)$.

\begin{lemma}\label{lemma:globallylinked}
Let $G = (V,E)$ be a strongly reconstructible graph in $\CC^d$ and suppose that a pair of vertices $u,v \in V$ 
is globally linked in $G$ in $\CC^d$.
Then $G' = G + uv$ is strongly reconstructible in $\CC^d$. Moreover, if $G$ is fully reconstructible in $\CC^d$, then so is $G'$.
\end{lemma}
\begin{proof}
By Theorem \ref{theorem:strongmeasurereconstructibility}, it is sufficient to show that if $M_{d,G'} =_\psi M_{d,H'}$ under some edge bijection $\varpsi$, where $H'$ is a graph on the same number of vertices as $G'$, then $\varpsi$ is induced by a graph isomorphism. Let $H$ denote $H' - \varpsi(uv)$. Then $M_{d,G} =_\psi M_{d,H}$ under the edge bijection $\varpsi|_{E(G)}$, and thus again by Theorem \ref{theorem:strongmeasurereconstructibility}, the strong reconstructibility of $G$ implies that $\varpsi|_{E(G)}$ is induced by a graph isomorphism $\varphi: V(G) \rightarrow V(H)$. It is sufficient to show that $\varpsi(uv) = \varphi(u)\varphi(v)$. After composing $\varpsi$ with the edge bijection induced by $\varphi^{-1}$, this amounts to showing that if $M_{d,G + uv} =_{\psi'} M_{d,G + u'v'}$ under the edge bijection $\psi'$ that fixes the edges of $G$ and sends $uv$ to $u'v'$, then $\{u',v'\} = \{u,v\}$. 

Let $(G,p)$ be a generic realization of $G$ in $\CC^d$. By Lemma \ref{lemma:genericimage}, there is a generic realization $(G,q)$, equivalent to $(G,p)$ and such that $m_{uv}(p) = m_{u'v'}(q)$. Since $\{u,v\}$ is globally linked in $G$, we must also have $m_{uv}(p) = m_{uv}(q)$. It follows from the genericity of $q$ that $\{u,v\} = \{u',v'\}$, as required. 

The same proof works when $G$ is fully reconstructible in $\CC^d$.
\end{proof}

We may also consider the effect of edge deletion on unlabeled reconstructibility. The analogue of Question \ref{question:linkedpair} for deleting edges is not true: it is not difficult to find an example of a graph $G$ that is strongly (or fully) reconstructible in $\CC^d$ but for which $G - uv$ is not strongly (or fully) reconstructible for some edge $uv \in E$, even though $r_d(G - uv) = r_d(G)$. This can happen e.g.\ if $G$ is globally rigid and $G - uv$ contains an $\Rd$-bridge. However, it is possible that the analogue of Lemma \ref{lemma:globallylinked} for edge deletions is true.

\begin{question}\label{question:edgedeletion}
Let $d \geq 3$ and let $G = (V,E)$ be strongly reconstructible in $\CC^d$ (fully reconstructible in $\CC^d$, respectively). Is it true that 
if for some edge $uv \in E$ we have that $\{u,v\}$ is globally linked in $G-uv$, then $G-uv$ is strongly reconstructible in $\CC^d$ (fully reconstructible in $\CC^d$, respectively)?
%\end{enumerate}
\end{question}

The characterization of strong and full reconstructibility in $\CC^1$ and $\CC^2$ given by Theorem \ref{theorem:GJmain} shows that for $d = 2$, the answer to Question \ref{question:edgedeletion} is positive, while for $d=1$, it is negative: let $G$ be a $3$-connected graph and suppose that $G - uv$ is not $3$-connected for some edge $uv \in E(G)$. Then $G$ is fully reconstructible in $\CC^1$ and $\{u,v\}$ is globally linked in $G-uv$
(in fact, $G-uv$ is globally rigid in $\RR^1$), but $G - uv$ is not strongly reconstructible in $\CC^d$.

\section{Graphs with nonseparable rigidity matroids}\label{section:mseparable}

In light of Theorem \ref{theorem:mconnected}, the combinatorial properties of $\Rd$-connected graphs may be of interest in studying global rigidity. However, not much seems to be known about these graphs in the $d \geq 3$ case. In this section, we collect three results related to this notion.

\begin{theorem}
\label{theorem:mconndown}
Let $G=(V,E)$ be an $\Rd$-connected graph.
Then $G$ is $\mathcal{R}_{d'}$-connected for all $1\leq d'\leq d$.
\end{theorem}

\begin{proof}
It suffices to consider the $d\geq 2$ case and show that $G$ is $\mathcal{R}_{d-1}$-connected.
We may also assume that $G$ is an $\Rd$-circuit.
Consider a generic realization $(G,p)$ of $G$ in $\RR^d$. Since $G$ is an $\Rd$-circuit, there exists a unique, (up to scalar multiplication) non-zero equilibrium stress $\omega = (\omega_e)_{e \in E}$ of $(G,p)$, which is
non-zero on every edge $e\in E$.

For a contradiction, suppose that $G$ is $\mathcal{R}_{d-1}$-separable and let
$E_1,E_2$ be a separation, that is, $E_1\cup E_2=E$ and $r_{d-1}(E_1)+r_{d-1}(E_2)=r_{d-1}(E)$.

Let $(G,p_i)$ be the $(d-1)$-dimensional realization of $G$ obtained from $(G,p)$ by
a projection along the $i$-axis, for $1\leq i\leq d$. These projected
frameworks are also generic in $\RR^{d-1}$ and $\omega$ is a stress on each $(G,p_i)$.
Let $R(G,p_i)$ be the matrix obtained from the rigidity matrix of $(G,p)$ by replacing
the $|V|$ columns corresponding to coordinate $i$ by all-zero columns. Thus the
rigidity matrix of $(G,p_i)$ can be obtained from $R(G,p_i)$ by removing these zero columns.

Since $E_1,E_2$ is a separation, we must have $\sum_{e\in E_1} \omega_e R_e(G,p_i)=0$ for all $1\leq i\leq d+1$,
where $R_e(G,p_i)$ is the row of the edge $e$ in $R(G,p_i)$.
This gives 
\begin{equation*}
    (d-1) \sum_{e\in E_1} \omega_e R_e(G,p) = \sum_{i=1}^{d} \sum_{e\in E_1} \omega_e R_e(G,p_i) = 0,
\end{equation*}
implying that the restriction of $\omega$ to $E_2$ is a non-zero stress on a proper subframework of $(G,p)$. 
Then extending this restricted stress to all of $E$ by setting it zero on every $e \in E_2$ gives another stress of $(G,p)$, contradicting the uniqueness of $\omega$. 
\end{proof} 

Our next result is a characterization of $\Rd$-connected cone graphs. 
We shall use the fact that $G$ is an $\Rd$-circuit if and only if $G^v$ is an $\mathcal{R}_{d+1}$-circuit. Although this result seems to be folklore, we could not find any proofs in the literature, so we provide one (due to W.\ Whiteley \cite{whpriv}) for completeness. For this, we shall need the following geometric result about coning, which is implicit in the work of Whiteley \cite{coningwhiteley} (see also \cite[Theorem 3]{symmetricconing}). 

\begin{theorem}\label{thm:coninggeometric}
Let $d \geq 1$ and let $G = (V,E)$ be a graph. Let $(G^v,p)$ be a realization of the cone graph $G^v$ in $\RR^{d+1}$ and let $H \subseteq \RR^{d+1}$ be a hyperplane not containing $p(v)$ and not parallel to any of the lines $\overline{p(v)p(u)}, u \in V$. Finally, let $(G,p_H)$ be a framework in $\RR^d$ obtained by projecting $p(u),u \in V$ onto $H$ from $p(v)$ and then identifying $H$ with $\RR^d$. Then the following holds:  \[\rank(R(G^v,p)) = \rank(R(G^v,p_H)) + |V|.\] This also implies that \[r_{d+1}(G^v) = r_d(G) + |V|.\]
\end{theorem}

\begin{lemma}\label{lemma:mcircuitcone}
Let $d \geq 1$ and let $G$ be a graph and $G^v$ its cone graph. Then $G$ is an $\Rd$-circuit if and only if $G^v$ is an $\mathcal{R}_{d+1}$-circuit.
\end{lemma}
\begin{proof}
If $G^v$ is an $\mathcal{R}_{d+1}$-circuit, then by Theorem \ref{thm:coninggeneric}(a) $G$ is not $\Rd$-independent. On the other hand, for every edge $uw \in E(G)$ we have that $(G - uw)^v = G^v - uw$ is $\mathcal{R}_{d+1}$-independent, so by the same theorem $G - uw$ is $\Rd$-independent. This shows that $G$ is an $\Rd$-circuit, as desired. 

Now let $G$ be an $\Rd$-circuit. Again by Theorem \ref{thm:coninggeneric}(a), $r_{d+1}(G^v) = |E(G^v)| - 1 = r_{d+1}(G^v - uw)$ for any edge $uw \in E(G)$. Thus, we only need to prove that $r_{d+1}(G^v - uv) = r_{d+1}(G^v)$ for any edge $uv$ incident to the cone vertex.

Let $uv \in E(G^v)$ be such an edge and consider a framework $(G^v,p)$ in which the vertices of $G$ lie in the $x_{d+1} = 0$ hyperplane in a generic position (when viewed as a framework in $\RR^d$) and $p(v)$ lies outside of this hyperplane. Let $w$ be a neighbour of $u$ in $G$. Theorem \ref{thm:coninggeometric} and the fact that the subframework $(G,p|_{V(G)})$ is generic imply that
\begin{equation*}
\rank(R(G^v,p)) = r_{d+1}(G^v) = r_{d+1}(G^v - uw) = \rank(R(G^v - uw,p)).    
\end{equation*}
Consider a framework $(G,q)$ obtained from $(G,p)$ by changing the last coordinate of $p(w)$ by a sufficiently small amount, so that $\rank(R(G^v - uw,q)) = r_{d+1}(G^v) = \rank(R(G^v,q))$ still holds. Then there is an equilibrium stress $\omega$ of $(G^v,q)$ that is non-zero on $uw$. This stress must be non-zero on $uv$ as well, since in this framework, the rest of the edges incident to $u$ all lie in the $x_{d+1} = 0$ hyperplane. This shows that 
\begin{equation*}
\rank(R(G^v - uv,q)) = \rank(R(G^v,q)) = r_{d+1}(G^v),   
\end{equation*}
which implies $r_{d+1}(G^v - uv) = r_{d+1}(G^v)$, as desired.
\end{proof}

%Given an $H$-graph $G$ in $\RR^d$, it is known that the {\it cone} of $G$ is an
%$H$-graph in $\RR^{d+1}$. (The cone of $G$ is obtained from $G$ by adding a new vertex
%$v$ and new edges from $v$ to every vertex of $G$.)
%The next lemma shows that coning cannot be used in a similar manner
%to transfer our new $H$-graphs to higher dimensions.  

\begin{theorem}\label{lemma:conemconnected}
%Let $G$ be redundantly rigid in $\RR^d$ and 
Let $G$ be a graph and let $G^v$ denote its cone graph. Then
$G^v$ is $\mathcal{R}_{d+1}$-connected if and only if
$G$ is connected and it has no $\Rd$-bridges.
\end{theorem}

\begin{proof}

For the ``only if'' direction, observe that coning takes an $\Rd$-bridge of $G$ to an $\mathcal{R}_{d+1}$-bridge of $G^v$, which follows from Theorem \ref{thm:coninggeometric},
%\footnote{In order to see this first note that the coning theorem implies that $G$ is rigid in $\RR^d$ if and only if $G^v$ is rigid in $\RR^{d+1}$. This fact easily implies that if $e$ is an $\Rd$-bridge in $G$, and $G$ is rigid in $\RR^d$, then $e$ is an $\mathcal{R}_{d+1}$-bridge in $G^v$. If $G$ is not rigid, then make it rigid by adding a set of
%$\Rd$-bridges and then apply the previous argument.}
and that an $\Rd$-connected graph on at least two edges has no $\Rd$-bridges. Moreover, $\Rd$-connected graphs are $2$-connected,\footnote{This follows e.g.\ from Theorem \ref{theorem:mconndown} by recalling that a graph is $\mathcal{R}_1$-connected if and only if it is $2$-connected.} while the cone graph of a disconnected graph is not.

To prove the ``if'' direction, 
let us first observe that for any edge $xy$ of $G$, $xy$ is in the same $\mathcal{R}_{d+1}$-connected component of $G^v$ as $vx$ and $vy$. Indeed, since $xy$ is not an $\Rd$-bridge in $G$, it is contained in some subgraph of $G$ that is an $\Rd$-circuit, and by Lemma \ref{lemma:mcircuitcone}, the cone of this subgraph is an $\mathcal{R}_{d+1}$-circuit which contains all three of these edges.

Thus it is sufficient to prove that any pair of cone edges $vx, vy$ is in the same $\mathcal{R}_{d+1}$-connected component of $G^v$. By assumption, there is a path $x = u_0,u_1,\dots,u_k = y$ in $G$. Now by the previous observation $vu_i$ and $vu_{i+1}$ are in the same $\mathcal{R}_{d+1}$-connected component of $G^v$, for all $0 \leq i < k$. By transitivity, we get that $vx$ and $vy$ are also in the same $\mathcal{R}_{d+1}$-connected component, as desired.
\end{proof}

 For a graph $G=(V,E)$ let $\dof_d(G)=d|V|-\binom{d+1}{2}-r_d(G)$ denote its
``degrees of freedom'' in the context of $d$-dimensional generic rigidity.
The next theorem verifies a general combinatorial
property of highly connected $\Rd$-separable graphs
and may be useful in the construction of further
families of examples. 

\begin{theorem}
\label{count}
Let $d \geq 1$ and let $G$ be a $(d+1)$-connected and 
redundantly rigid graph in $\RR^d$.
Suppose that $G$ is $\Rd$-separable and let $H_1$, $H_2$,...,$H_q$ be the
$\Rd$-connected components of $G$.
Then $$\sum_{i=1}^q \dof_d(H_i) \geq \binom{d+1}{2}.$$
\end{theorem}

\begin{proof}
Let $X_i=V(H_i)-\cup_{j\not= i}V(H_j)$
denote the set of vertices belonging to no other $\Rd$-connected component than $H_i$,
and let $Y_i=V(H_i)-X_i$ for $1\leq i\leq q$.
Let
$n_i=|V(H_i)|$, $x_i=|X_i|$, $y_i=|Y_i|$. Clearly,
$n_i=x_i+y_i$ and
$|V|=\sum_{i=1}^q x_i + |\cup_{i=1}^q Y_i|$. Moreover, we have
$\sum_{i=1}^q y_i \geq 2|\cup_{i=1}^q Y_i|$.
Since $G$ is redundantly rigid,
every edge of $G$ is in some $\Rd$-circuit. Every $\Rd$-circuit
has at least $d+2$ vertices. Thus we have that $n_i\geq d+2$ for
$1\leq i\leq q$. Furthermore, since
$G$ is $(d+1)$-connected,
%$y_i\geq 2$ for all $1\leq i\leq q$, and
$y_i\geq (d+1)$ holds for all $\Rd$-components.

Let us choose a base $B_i$ in each rigidity matroid ${\cal R}(H_i)$.
Using the above inequalities we have
\begin{align*}
d|V|-\binom{d+1}{2}&= |\cup_{i=1}^q B_i|=\sum_{i=1}^q |B_i|
=\sum_{i=1}^q \left(dn_i-\binom{d+1}{2}-\dof_d(H_i)\right) \\
&=d\sum_{i=1}^q n_i - \binom{d+1}{2}q - \sum_{i=1}^q \dof_d(H_i) \\
&=\left(d\sum_{i=1}^q x_i + \frac{d}{2}\sum_{i=1}^q y_i\right)
+ \left(\frac{d}{2}\sum_{i=1}^q y_i-\binom{d+1}{2}q\right) - \sum_{i=1}^q \dof_d(H_i) \\
&\geq d|V| - \sum_{i=1}^q \dof_d(H_i)
\end{align*}
Thus we must have $\sum_{i=1}^q \dof_d(H_i) \geq \binom{d+1}{2}$, as claimed.
\end{proof}

The graph of Figure \ref{figure:mseparable2}(a) shows  that Theorem \ref{count} is, in some sense, tight: it has a unique non-rigid $\mathcal{R}_3$-connected component with six degrees of freedom. 
Note that for $d=1,2$ the $\Rd$-connected components of a graph are rigid. Thus the theorem implies that for $d\leq 2$ the $(d+1)$-connected redundantly rigid graphs
are $\Rd$-connected, which was shown in \cite[Theorem 3.2]{JJconnrig} by a similar argument.

\section{Acknowledgements}

This work was supported by the Hungarian Scientific Research Fund grant no. K 135421,
%This work was supported by the Hungarian Scientific Research Fund grants no. K 109240, K135421,
and the project Application Domain Specific Highly Reliable IT Solutions which has been
implemented with the support provided from the National Research, Development and
Innovation Fund of Hungary, financed under the Thematic Excellence Programme 
TKP2020-NKA-06 (National Challenges Subprogramme) funding scheme.

\end{document}